\documentclass{amsart}
\usepackage[dvipdfmx]{graphicx}
\usepackage[varg]{txfonts}
\usepackage{mathtools,amssymb,amsthm,enumerate,comment}
\usepackage[all]{xy}
\usepackage{fullpage}
\theoremstyle{plain}

\newtheorem{thm}{Theorem}[section]
\newtheorem{lem}[thm]{Lemma}
\newtheorem{prop}[thm]{Proposition}

\newtheorem{cor}[thm]{Corollary}

\theoremstyle{definition}

\newtheorem{defn}[thm]{Definition}
\newtheorem{rem}[thm]{Remark}
\newtheorem{ex}[thm]{Example}

\newcommand{\R}{\mathbb{R}}
\newcommand{\Z}{\mathbb{Z}}

\newcommand{\K}{\mathcal{K}}

\newcommand{\abs}[1]{\left\lvert {#1} \right\rvert}
\newcommand{\ang}[1]{\left\langle {#1} \right\rangle}

\newcommand{\Conf}{\mathrm{Conf}}

\newcommand{\DR}{\mathit{DR}}

\newcommand{\FH}{\mathit{FH}}
\newcommand{\fK}{\widetilde{\K}}

\newcommand{\id}{\mathrm{id}}

\newcommand{\lk}{\mathrm{lk}} 

\newcommand{\pr}{\mathrm{pr}}

\newcommand{\SO}{\mathrm{SO}}

\newcommand{\supp}{\mathrm{supp}} 

\newcommand{\wtf}{\widetilde{f}} 

\newcommand{\vf}{\mathrm{v}_\mathrm{f}}
\newcommand{\vi}{\mathrm{v}_\mathrm{i}}
\newcommand{\vol}{\mathrm{vol}}

\numberwithin{equation}{section}
\numberwithin{figure}{section}

\allowdisplaybreaks[1]

\title{The Fox-Hatcher cycle and a Vassiliev invariant of order three}
\author{Saki Kanou and Keiichi Sakai}
\thanks{The 2nd author is partially supported by JSPS KAKENHI Grant Number 20K03608.}
\address{Faculty of Mathematics, Shinshu University, 3-1-1 Asahi, Matsumoto, Nagano 390-8621, Japan}
\email{20ss104f@gmail.com}
\email{ksakai@math.shinshu-u.ac.jp}
\date{\today}

\begin{document}
\begin{abstract}
We show that the integration of a 1-cocycle $I(X)$ of the space of long knots in $\R^3$ over the \emph{Fox-Hatcher 1-cycles} gives rise to a Vassiliev invariant of order exactly three.
This result can be seen as a continuation of the previous work of the second named author \cite{K09}, proving that the integration of $I(X)$ over the \emph{Gramain 1-cycles} is the \emph{Casson invariant}, the unique nontrivial Vassiliev invariant of order two (up to scalar multiplications).
The result in the present paper is also analogous to part of Mortier's result \cite{Mortier15}.
Our result differs from, but is motivated by, Mortier's one in that the 1-cocycle $I(X)$ is given by the \emph{configuration space integrals} associated with graphs while Mortier's cocycle is obtained in a combinatorial way.
\end{abstract}
\maketitle

\section{Introduction}
The spaces of smooth embeddings of manifolds are receiving a lot of attention in topology, on the ground that various important methods in algebraic and geometric topology are being applied to the spaces.
In this paper we study the space of \emph{(framed) long knots} in $\R^3$.

\begin{defn}
A \emph{long knot} is an embedding $f\colon\R^1\hookrightarrow\R^3$ satisfying $f(x)=(x,0,0)$ for any $x\in\R^1$ with $\abs{x}\ge 1$. A \emph{framed long knot} is a smooth map $\wtf=(f,w)\colon\R^1\to\R^3\times\SO(3)$ such that $f$ is a long knot, the first column of $w(x)\in\SO(3)$ is equal to $f'(x)/\abs{f'(x)}$ and $w(x)$ is the identity matrix for any $x\in\R^1$ with $\abs{x}\ge 1$.
The space of all long knots (respectively framed long knots) is denoted by $\K$ (respectively $\fK$).
\end{defn}

The recent studies of $\K$ (and its high dimensional analogues) are revealing relations between the topological nature of $\K$ and the \emph{Vassiliev invariants} (see for example \cite{JacksonMoffatt19}) for knots and links.
In \cite{K09} the second named author has constructed a de Rham 1-cocycle $I(X)$ of $\K$ (see \S\ref{s:cocycle}), by means of the integrations over configuration spaces associated with graphs, and has shown that the integration of $I(X)$ over the \emph{Gramain cycles} of $\K$ gives rise to the \emph{Casson invariant} $v_2$, the Vassiliev invariant of order two uniquely characterized by $v_2(\text{trivial knot})=0$ and $v_2(\text{trefoil knot})=1$.
This may be seen as a real valued version of \cite[Theorem 2]{Turchin06}.
After that Mortier has given another 1-cocyle $\alpha^1_3$ of $\K$ in a combinatorial way and has shown that its evaluations over the Gramain cycles and the \emph{Fox-Hatcher cycles} $\FH$ are Vassiliev invariants of orders respectively two and three \cite[Theorem~4.1]{Mortier15}.
In \cite{Fiedler20} 1-cocycles on $\K$ are also studied in detail from a combinatorial viewpoint.

The main result in the present paper is analogous to the order three part of Mortier's result.

\begin{thm}\label{thm:main}
The integration of $I(X)$ over the Fox-Hatcher cycles gives rise to a Vassiliev invariant of order three for framed long knots.
More precisely we have
\begin{equation}\label{eq:invariant_explicit_form}
 \int_{p_*\FH_{\widetilde{f}}}I(X) = 6v_3(f)-\lk(\widetilde{f})v_2(f),
\end{equation}
where
\begin{itemize}
\item
	$p\colon\fK\to\K$ is the first projection and $f=p(\wtf)$,
\item
	$v_2$ is the Casson invariant, and $v_3$ is the Vassiliev invariant of order three characterized by the conditions
	\begin{equation}\label{eq:v_3_characterization}
	 v_3(\text{trivial knot})=0,\quad
	 v_3(3_1^+)=1,\quad
	 v_3(3_1^-)=-1
	\end{equation}
	($3_1^+$ and $3_1^-$ are respectively the right-handed and the left-handed trefoil knots),
	and
\item
	$\lk(\widetilde{f})\in\Z$ is the \emph{framing number} of $\widetilde{f}$ (see Remark~\ref{rem:framing_number} below).
\end{itemize}
\end{thm}

\begin{rem}\label{rem:framing_number}
The framing number $\lk(\wtf)$ is the linking number of $f=p(\wtf)$ and $f'$, where $f'$ is the long knot obtained by moving $f$ slightly into the direction of the second column of $w$.
In fact the map $p\times\lk\colon\fK\to\K\times\Z$ is a homotopy equivalence \cite[Proposition 9]{Budney03}, and the framing number uniquely determines the framing $w$ up to homotopy.
Thus we may regard a framed long knot as a pair $(f,w)$ of $f\in\K$ and $w\in\Z$.
\end{rem}

The $1$-cocycle $I(X)$ is constructed by means of the \emph{configuration space integral associated with graphs}, that was developed in \cite{AltschulerFreidel97,BottTaubes94,Kohno94} to describe Vassiliev invariants and was generalized in \cite{CCL02} to obtain a cochain map from a graph complex to $\Omega^*_{\DR}(\K)$ (up to some correction terms, that vanish in the cases of the spaces of long knots in high dimensional spaces).
Vassiliev invarians ($0$-cocycles of $\K$) are obtained from trivalent graphs, while our $1$-cocycle $I(X)$ comes from non-trivalent graphs (see Figure~\ref{fig:Nontrivalent_graph}).
It is very interesting, although not strange, that non-trivalent graphs may also have information of Vassiliev invariants.

We note that the right hand side of \eqref{eq:invariant_explicit_form} coincides with the formula for Mortier's invariant of order three. We thus expect that the 1-cocycle $I(X)$ is cohomologous to Mortier's $\alpha^1_3$ (this is true on the connected components of torus and hyperbolic knots; see \cite[p.~2]{Hatcher02}).

This paper is organized as follows.
In \S\ref{s:FH_cycle} the Fox-Hatcher cycle is introduced, and in \S\ref{s:cocycle} the construction of the 1-cocycle $I(X)$ is reviewed.
Our invariant $v$, the left hand side of \eqref{eq:invariant_explicit_form}, is shown to be of order three in Corollary~\ref{cor:v_order_3}.
The key ingredient is Theorem~\ref{thm:D^3v} and is proved in \S\ref{ss:D^3v}.
The formula \eqref{eq:invariant_explicit_form} is verified in \S\ref{ss:explicit_form}.

\section{The Fox-Hatcher cycle}\label{s:FH_cycle}
\subsection{The Fox-Hatcher cycle}
The Fox-Hatcher cycle
was introduced in \cite{Fox66}, and was later studied in \cite{Hatcher02} from the viewpoint of the space of knots.
If $f=p(\wtf)$ is not trivial, it then gives a non-zero element of $\pi_1(\fK_{\wtf })$, where $\fK_{\wtf}$ is the path component of $\fK$ containing $\wtf $.

The Fox-Hatcher cycle
is defined as follows.
A framed long knot can be seen as a based embedding $f\colon S^1\hookrightarrow S^3$ (we see $S^3$ as in $\R^4$) together with a framing $w$, with a prescribed behavior near the basepoint.
For $t\in S^1$, $w(t)$ is an orthonormal basis of $T_{f(t)}S^3$ whose first vector is $f'(t)/\abs{f'(t)}$.
There exists an $S^1$-action on the space of such embeddings defined by $(\theta\cdot(f,w))(t)\coloneqq(A(\theta)^{-1}f(t-\theta),A(\theta)^{-1}w(t-\theta))$, where $A(\theta)\in\SO(4)$ is the matrix given by $A(\theta)=(w(\theta),f(\theta))$.
For any $\wtf\in\fK$, this action determines a 1-cycle $\FH_{\wtf}\colon S^1\to\fK_{\wtf}$ and we call it the \emph{Fox-Hatcher cycle}.
We notice that the $S^1$-action looks very similar to the natural $S^1$-action on free loop spaces by the reparametrization, and in fact this action defines a \emph{BV-operation} on $H_*(\fK)$ \cite{K10}.

Practically it is convenient to describe $\FH$ on knot diagrams.
In this paper a framed long knot is drawn in a usual knot diagram with so-called \emph{blackboard framing}.

\begin{defn}
Let $D$ be a knot diagram of $\wtf$ with blackboard framing and $c$ the ``left-most'' crossing, namely the crossing that we meet first when traveling from $f(-1)$ along the natural orientation of $f$.
We call the transformation shown in Figure~\ref{fig:FH_1} the \emph{Fox-Hatcher move} (FH-move for short) \emph{on $c$}.
\end{defn}
\begin{figure}
\centering
{\unitlength 0.1in%
\begin{picture}(43.3700,10.0000)(2.6300,-12.0000)%
%
\special{pn 13}%
\special{pa 300 700}%
\special{pa 650 700}%
\special{fp}%
%
\special{pn 13}%
\special{pa 750 700}%
\special{pa 900 700}%
\special{fp}%
\special{sh 1}%
\special{pa 900 700}%
\special{pa 833 680}%
\special{pa 847 700}%
\special{pa 833 720}%
\special{pa 900 700}%
\special{fp}%
\put(12.0000,-7.0000){\makebox(0,0){$\cdots$}}%
%
\special{pn 13}%
\special{ar 1100 700 400 400 1.5707963 4.7123890}%
\put(13.0000,-3.0000){\makebox(0,0){$\cdots$}}%
\put(13.0000,-11.0000){\makebox(0,0){$\cdots$}}%
%
\special{pn 13}%
\special{pa 1400 700}%
\special{pa 2000 700}%
\special{fp}%
\special{sh 1}%
\special{pa 2000 700}%
\special{pa 1933 680}%
\special{pa 1947 700}%
\special{pa 1933 720}%
\special{pa 2000 700}%
\special{fp}%
\put(24.0000,-7.0000){\makebox(0,0){$\rightsquigarrow$}}%
%
\special{pn 13}%
\special{pa 2800 700}%
\special{pa 3500 700}%
\special{fp}%
\special{sh 1}%
\special{pa 3500 700}%
\special{pa 3433 680}%
\special{pa 3447 700}%
\special{pa 3433 720}%
\special{pa 3500 700}%
\special{fp}%
\put(38.0000,-7.0000){\makebox(0,0){$\cdots$}}%
%
\special{pn 13}%
\special{ar 3700 250 50 50 1.5707963 4.7123890}%
%
\special{pn 13}%
\special{ar 3700 1150 50 50 1.5707963 4.7123890}%
\put(39.0000,-3.0000){\makebox(0,0){$\cdots$}}%
\put(39.0000,-11.0000){\makebox(0,0){$\cdots$}}%
%
\special{pn 13}%
\special{ar 3700 700 500 500 4.7123890 1.5707963}%
%
\special{pn 13}%
\special{pa 4000 700}%
\special{pa 4150 700}%
\special{fp}%
%
\special{pn 13}%
\special{pa 4250 700}%
\special{pa 4600 700}%
\special{fp}%
\special{sh 1}%
\special{pa 4600 700}%
\special{pa 4533 680}%
\special{pa 4547 700}%
\special{pa 4533 720}%
\special{pa 4600 700}%
\special{fp}%
\put(5.0000,-7.0000){\makebox(0,0){$\bullet$}}%
\put(18.0000,-7.0000){\makebox(0,0){$\bullet$}}%
\put(5.0000,-8.0000){\makebox(0,0){$f(-1)$}}%
\put(18.0000,-8.0000){\makebox(0,0){$f(1)$}}%
\put(44.0000,-8.0000){\makebox(0,0){$f(1)$}}%
\put(44.0000,-7.0000){\makebox(0,0){$\bullet$}}%
\put(30.0000,-7.0000){\makebox(0,0){$\bullet$}}%
\put(30.0000,-8.0000){\makebox(0,0){$f(-1)$}}%
\put(7.5000,-6.5000){\makebox(0,0)[lb]{$c$}}%
\put(41.5000,-6.5000){\makebox(0,0)[rb]{$c'$}}%
%
\special{pn 8}%
\special{pn 8}%
\special{pa 3700 1100}%
\special{pa 3692 1100}%
\special{fp}%
\special{pa 3670 1099}%
\special{pa 3663 1098}%
\special{fp}%
\special{pa 3642 1096}%
\special{pa 3634 1095}%
\special{fp}%
\special{pa 3613 1090}%
\special{pa 3605 1088}%
\special{fp}%
\special{pa 3585 1083}%
\special{pa 3578 1081}%
\special{fp}%
\special{pa 3559 1074}%
\special{pa 3552 1072}%
\special{fp}%
\special{pa 3533 1064}%
\special{pa 3526 1060}%
\special{fp}%
\special{pa 3507 1051}%
\special{pa 3501 1047}%
\special{fp}%
\special{pa 3483 1036}%
\special{pa 3477 1032}%
\special{fp}%
\special{pa 3460 1020}%
\special{pa 3454 1015}%
\special{fp}%
\special{pa 3438 1002}%
\special{pa 3433 998}%
\special{fp}%
\special{pa 3419 985}%
\special{pa 3414 979}%
\special{fp}%
\special{pa 3400 964}%
\special{pa 3396 959}%
\special{fp}%
\special{pa 3383 944}%
\special{pa 3378 937}%
\special{fp}%
\special{pa 3366 921}%
\special{pa 3363 915}%
\special{fp}%
\special{pa 3351 896}%
\special{pa 3348 890}%
\special{fp}%
\special{pa 3338 871}%
\special{pa 3335 864}%
\special{fp}%
\special{pa 3327 845}%
\special{pa 3325 838}%
\special{fp}%
\special{pa 3318 818}%
\special{pa 3315 810}%
\special{fp}%
\special{pa 3310 790}%
\special{pa 3309 783}%
\special{fp}%
\special{pa 3305 762}%
\special{pa 3304 754}%
\special{fp}%
\special{pa 3301 733}%
\special{pa 3301 725}%
\special{fp}%
\special{pa 3300 703}%
\special{pa 3300 695}%
\special{fp}%
\special{pa 3301 674}%
\special{pa 3301 666}%
\special{fp}%
\special{pa 3304 645}%
\special{pa 3305 637}%
\special{fp}%
\special{pa 3309 616}%
\special{pa 3310 609}%
\special{fp}%
\special{pa 3316 589}%
\special{pa 3318 581}%
\special{fp}%
\special{pa 3325 561}%
\special{pa 3327 554}%
\special{fp}%
\special{pa 3336 535}%
\special{pa 3339 528}%
\special{fp}%
\special{pa 3348 509}%
\special{pa 3352 503}%
\special{fp}%
\special{pa 3363 485}%
\special{pa 3367 479}%
\special{fp}%
\special{pa 3378 463}%
\special{pa 3382 457}%
\special{fp}%
\special{pa 3395 441}%
\special{pa 3400 436}%
\special{fp}%
\special{pa 3414 421}%
\special{pa 3418 416}%
\special{fp}%
\special{pa 3433 402}%
\special{pa 3438 398}%
\special{fp}%
\special{pa 3454 385}%
\special{pa 3460 380}%
\special{fp}%
\special{pa 3476 369}%
\special{pa 3482 365}%
\special{fp}%
\special{pa 3500 353}%
\special{pa 3507 350}%
\special{fp}%
\special{pa 3525 340}%
\special{pa 3532 337}%
\special{fp}%
\special{pa 3551 329}%
\special{pa 3558 326}%
\special{fp}%
\special{pa 3578 319}%
\special{pa 3585 317}%
\special{fp}%
\special{pa 3606 311}%
\special{pa 3614 309}%
\special{fp}%
\special{pa 3634 305}%
\special{pa 3641 304}%
\special{fp}%
\special{pa 3663 302}%
\special{pa 3670 301}%
\special{fp}%
\special{pa 3692 300}%
\special{pa 3700 300}%
\special{fp}%
\end{picture}}%
\caption{The Fox-Hatcher move on $c$}
\label{fig:FH_1}
\end{figure}
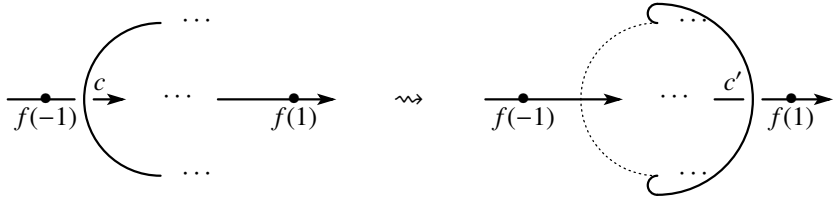
The left-most crossing $c$ disappears after the FH-move on $c$ and the right-most crossing $c'$ is created.
If the arc that moves in the FH-move is the over-arc (resp.\ under-arc) at $c$, then after the FH-move it becomes the over-arc (resp.\ under-arc) at $c'$.
We arrive the original diagram $D$ after performing the FH-moves for all the other crossings $c$ of $D$ and the newborn crossings $c'$.
The sequence of these FH-moves realizes $\FH_{\wtf}$.

\subsection{FH moves and Gauss diagrams}
The configuration of crossings of a knot diagram is encoded by \emph{(linear) Gauss diagrams}.
Here we see how the FH-move on the left-most crossing changes the Gauss diagram.

\begin{defn}\label{def:chord_diagram}
A \emph{(linear) Gauss diagram} is a subdivision of $\{1,2,\dots,2n\}$ for some natural number $n$ into a union $\bigcup_{1\le k\le n}\{i_k,j_k\}$ of $n$ subsets of cardinarity $2$.
\end{defn}

A Gauss diagram can be seen as a graph on $\R^1$ with even number of vertices all of which are on $\R^1$ and each vertex is joined by exactly one edge with another vertex.
Here segments in $\R^1$ interposed between two vertices are not regarded as edges.
See Figure~\ref{fig:Gauss} for example.
\begin{figure}
\centering
\input{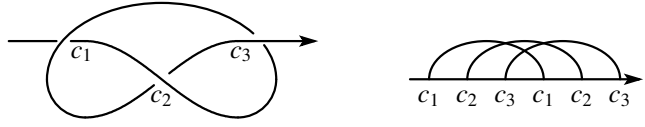}
\caption{A knot diagram and its Gauss diagram}
\label{fig:Gauss}
\end{figure}

\begin{defn}[{\cite[Definition~3.3]{K09}}]\label{def:crossings_respecting_Gauss_diagram}
Let $c_1,\dots,c_n$ be (part of the) crossings of a knot diagram of $f\in\K$ such that each $c_i$ corresponds to $f(p_i)$ and $f(q_i)$, with $-1<p_1<\dots<p_n<1$ and $p_i<q_i$ for any $i=1,\dots,n$.
We say that the crossings $c_1,\dots,c_n$ \emph{respect} a Gauss diagram $G$ if $G$ is isomorphic to the Gauss diagram $G_{c_1,\dots,c_n}$ obtained by joining $p_i$ and $q_i$ for $i=1,\dots,n$.
See Figure~\ref{fig:Gauss}.
\end{defn}

Under the setting of Definition~\ref{def:crossings_respecting_Gauss_diagram}, the left-most crossing is $c_1$.
Let $G$ be the Gauss diagram that $c_1,\dots,c_n$ respect.
Then the new knot diagram obtained by performing the FH-move on $c_1$ has crossings $c_2,\dots,c_n,c_1'$ that respect the Gauss diagram $G'$ obtained by moving the left-most vertex (corresponding to $c_1$) to the right-most one.
See Figure~\ref{fig:Gauss_FH}.
\begin{figure}
\centering
{\unitlength 0.1in%
\begin{picture}(42.0000,5.2700)(2.0000,-8.2700)%
%
\special{pn 13}%
\special{pa 600 500}%
\special{pa 1000 500}%
\special{pa 1000 700}%
\special{pa 600 700}%
\special{pa 600 500}%
\special{pa 1000 500}%
\special{fp}%
%
\special{pn 13}%
\special{pa 200 600}%
\special{pa 600 600}%
\special{fp}%
%
\special{pn 13}%
\special{pa 1000 600}%
\special{pa 1400 600}%
\special{fp}%
%
\special{pn 13}%
\special{pa 1400 500}%
\special{pa 1800 500}%
\special{pa 1800 700}%
\special{pa 1400 700}%
\special{pa 1400 500}%
\special{pa 1800 500}%
\special{fp}%
%
\special{pn 13}%
\special{pa 1800 600}%
\special{pa 2200 600}%
\special{fp}%
\special{sh 1}%
\special{pa 2200 600}%
\special{pa 2133 580}%
\special{pa 2147 600}%
\special{pa 2133 620}%
\special{pa 2200 600}%
\special{fp}%
%
\special{pn 13}%
\special{ar 800 600 400 300 3.1415927 6.2831853}%
\put(12.0000,-9.0000){\makebox(0,0){$G$}}%
\put(4.0000,-7.0000){\makebox(0,0){$c_1$}}%
\put(12.0000,-7.0000){\makebox(0,0){$c_1$}}%
\put(24.0000,-5.0000){\makebox(0,0){$\rightsquigarrow$}}%
%
\special{pn 13}%
\special{pa 2600 600}%
\special{pa 2800 600}%
\special{fp}%
%
\special{pn 13}%
\special{pa 2800 500}%
\special{pa 3200 500}%
\special{pa 3200 700}%
\special{pa 2800 700}%
\special{pa 2800 500}%
\special{pa 3200 500}%
\special{fp}%
%
\special{pn 13}%
\special{ar 3800 600 400 300 3.1415927 6.2831853}%
%
\special{pn 13}%
\special{pa 3200 600}%
\special{pa 3600 600}%
\special{fp}%
%
\special{pn 13}%
\special{pa 3600 500}%
\special{pa 4000 500}%
\special{pa 4000 700}%
\special{pa 3600 700}%
\special{pa 3600 500}%
\special{pa 4000 500}%
\special{fp}%
%
\special{pn 13}%
\special{pa 4000 600}%
\special{pa 4400 600}%
\special{fp}%
\special{sh 1}%
\special{pa 4400 600}%
\special{pa 4333 580}%
\special{pa 4347 600}%
\special{pa 4333 620}%
\special{pa 4400 600}%
\special{fp}%
\put(34.0000,-7.0000){\makebox(0,0){$c_1'$}}%
\put(42.0000,-7.0000){\makebox(0,0){$c_1'$}}%
\put(34.0000,-9.0000){\makebox(0,0){$G'$}}%
%
\special{pn 8}%
\special{ar 1100 500 400 200 3.1415927 6.2831853}%
%
\special{pn 8}%
\special{ar 1300 500 400 200 3.1415927 6.2831853}%
%
\special{pn 8}%
\special{ar 3300 500 400 200 3.1415927 6.2831853}%
%
\special{pn 8}%
\special{ar 3500 500 400 200 3.1415927 6.2831853}%
\end{picture}}%
\caption{The FH-move on $c_1$ on the Gauss diagram}
\label{fig:Gauss_FH}
\end{figure}
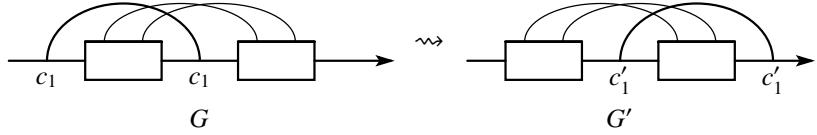

We eventually arrive the original Gauss diagram after performing the FH-moves on all the crossings $c$ of the original diagram and the newborn crossings $c'$.
This sequence produces a cycle of Gauss diagrams (see Figures~\ref{fig:Gauss_cycle_1}, \ref{fig:Gauss_cycle_2}, \ref{fig:Gauss_cycle_3}).
In this way the set of all the Gauss diagrams is decomposed into the disjoint cycles.

\section{The cocycle $I(X)$}\label{s:cocycle}
In this section we give a quick review of the construction of differential forms on $\K$ associated with graphs.
See also \cite{AltschulerFreidel97,BottTaubes94,CCL02,Kohno94,Volic05} for details.

By a \emph{graph} we mean the oriented real line $\R^1$ together with two kind of vertices, one is called \emph{interval} and the other \emph{free}, and \emph{oriented edges} connecting them (see Figure~\ref{fig:graph}).
\begin{figure}
\centering
{\unitlength 0.1in%
\begin{picture}(16.3700,6.0000)(5.6300,-8.2700)%
%
\special{pn 13}%
\special{pa 600 800}%
\special{pa 2200 800}%
\special{fp}%
\special{sh 1}%
\special{pa 2200 800}%
\special{pa 2133 780}%
\special{pa 2147 800}%
\special{pa 2133 820}%
\special{pa 2200 800}%
\special{fp}%
\put(8.0000,-8.0000){\makebox(0,0){$\bullet$}}%
\put(16.0000,-8.0000){\makebox(0,0){$\bullet$}}%
%
\special{pn 13}%
\special{pn 13}%
\special{pa 1200 600}%
\special{pa 1213 600}%
\special{fp}%
\special{pa 1230 601}%
\special{pa 1243 601}%
\special{fp}%
\special{pa 1259 602}%
\special{pa 1272 603}%
\special{fp}%
\special{pa 1289 605}%
\special{pa 1301 607}%
\special{fp}%
\special{pa 1318 609}%
\special{pa 1331 611}%
\special{fp}%
\special{pa 1347 614}%
\special{pa 1360 616}%
\special{fp}%
\special{pa 1376 621}%
\special{pa 1388 624}%
\special{fp}%
\special{pa 1404 628}%
\special{pa 1417 632}%
\special{fp}%
\special{pa 1433 638}%
\special{pa 1445 642}%
\special{fp}%
\special{pa 1460 648}%
\special{pa 1472 654}%
\special{fp}%
\special{pa 1487 661}%
\special{pa 1498 667}%
\special{fp}%
\special{pa 1512 675}%
\special{pa 1523 682}%
\special{fp}%
\special{pa 1537 692}%
\special{pa 1547 700}%
\special{fp}%
\special{pa 1559 712}%
\special{pa 1568 721}%
\special{fp}%
\special{pa 1578 734}%
\special{pa 1585 745}%
\special{fp}%
\special{pa 1592 759}%
\special{pa 1596 771}%
\special{fp}%
\special{pa 1599 787}%
\special{pa 1600 800}%
\special{fp}%
\put(8.0000,-4.0000){\makebox(0,0){$\bullet$}}%
%
\special{pn 13}%
\special{pa 800 800}%
\special{pa 800 600}%
\special{dt 0.030}%
\special{sh 1}%
\special{pa 800 600}%
\special{pa 780 667}%
\special{pa 800 653}%
\special{pa 820 667}%
\special{pa 800 600}%
\special{fp}%
%
\special{pn 13}%
\special{pa 800 600}%
\special{pa 800 400}%
\special{dt 0.030}%
\put(10.0000,-8.0000){\makebox(0,0){$\bullet$}}%
%
\special{pn 13}%
\special{pa 800 400}%
\special{pa 900 600}%
\special{dt 0.030}%
\special{sh 1}%
\special{pa 900 600}%
\special{pa 888 531}%
\special{pa 876 552}%
\special{pa 852 549}%
\special{pa 900 600}%
\special{fp}%
%
\special{pn 13}%
\special{pa 900 600}%
\special{pa 1000 800}%
\special{dt 0.030}%
%
\special{pn 13}%
\special{pn 13}%
\special{pa 800 800}%
\special{pa 801 787}%
\special{fp}%
\special{pa 804 771}%
\special{pa 808 760}%
\special{fp}%
\special{pa 815 745}%
\special{pa 822 734}%
\special{fp}%
\special{pa 833 721}%
\special{pa 841 712}%
\special{fp}%
\special{pa 853 700}%
\special{pa 863 692}%
\special{fp}%
\special{pa 877 683}%
\special{pa 887 675}%
\special{fp}%
\special{pa 902 667}%
\special{pa 913 661}%
\special{fp}%
\special{pa 928 653}%
\special{pa 940 648}%
\special{fp}%
\special{pa 955 642}%
\special{pa 967 637}%
\special{fp}%
\special{pa 983 632}%
\special{pa 996 628}%
\special{fp}%
\special{pa 1012 623}%
\special{pa 1024 620}%
\special{fp}%
\special{pa 1041 617}%
\special{pa 1053 614}%
\special{fp}%
\special{pa 1069 611}%
\special{pa 1082 609}%
\special{fp}%
\special{pa 1098 606}%
\special{pa 1111 605}%
\special{fp}%
\special{pa 1128 603}%
\special{pa 1141 602}%
\special{fp}%
\special{pa 1157 601}%
\special{pa 1170 601}%
\special{fp}%
\special{pa 1187 600}%
\special{pa 1200 600}%
\special{fp}%
%
\special{pn 13}%
\special{pa 1180 600}%
\special{pa 1200 600}%
\special{dt 0.030}%
\special{sh 1}%
\special{pa 1200 600}%
\special{pa 1133 580}%
\special{pa 1147 600}%
\special{pa 1133 620}%
\special{pa 1200 600}%
\special{fp}%
\put(14.0000,-4.0000){\makebox(0,0){$\bullet$}}%
\put(12.0000,-8.0000){\makebox(0,0){$\bullet$}}%
%
\special{pn 13}%
\special{pa 800 400}%
\special{pa 1100 400}%
\special{dt 0.030}%
\special{sh 1}%
\special{pa 1100 400}%
\special{pa 1033 380}%
\special{pa 1047 400}%
\special{pa 1033 420}%
\special{pa 1100 400}%
\special{fp}%
%
\special{pn 13}%
\special{pa 1100 400}%
\special{pa 1400 400}%
\special{dt 0.030}%
%
\special{pn 13}%
\special{pa 1400 400}%
\special{pa 1300 600}%
\special{dt 0.030}%
\special{sh 1}%
\special{pa 1300 600}%
\special{pa 1348 549}%
\special{pa 1324 552}%
\special{pa 1312 531}%
\special{pa 1300 600}%
\special{fp}%
%
\special{pn 13}%
\special{pa 1300 600}%
\special{pa 1200 800}%
\special{dt 0.030}%
%
\special{pn 13}%
\special{pa 1400 400}%
\special{pa 1600 600}%
\special{dt 0.030}%
\special{sh 1}%
\special{pa 1600 600}%
\special{pa 1567 539}%
\special{pa 1562 562}%
\special{pa 1539 567}%
\special{pa 1600 600}%
\special{fp}%
%
\special{pn 13}%
\special{pa 1600 600}%
\special{pa 1800 800}%
\special{dt 0.030}%
\put(18.0000,-8.0000){\makebox(0,0){$\bullet$}}%
\put(20.0000,-8.0000){\makebox(0,0){$\bullet$}}%
%
\special{pn 13}%
\special{pn 13}%
\special{pa 2100 700}%
\special{pa 2099 712}%
\special{fp}%
\special{pa 2096 727}%
\special{pa 2093 738}%
\special{fp}%
\special{pa 2085 752}%
\special{pa 2079 761}%
\special{fp}%
\special{pa 2071 770}%
\special{pa 2063 777}%
\special{fp}%
\special{pa 2052 785}%
\special{pa 2042 791}%
\special{fp}%
\special{pa 2029 796}%
\special{pa 2018 798}%
\special{fp}%
\special{pa 2003 800}%
\special{pa 1990 800}%
\special{fp}%
\special{pa 1976 797}%
\special{pa 1964 793}%
\special{fp}%
\special{pa 1950 787}%
\special{pa 1942 781}%
\special{fp}%
\special{pa 1931 773}%
\special{pa 1924 765}%
\special{fp}%
\special{pa 1916 755}%
\special{pa 1911 745}%
\special{fp}%
\special{pa 1905 732}%
\special{pa 1902 721}%
\special{fp}%
\special{pa 1900 705}%
\special{pa 1900 692}%
\special{fp}%
\special{pa 1903 677}%
\special{pa 1906 665}%
\special{fp}%
\special{pa 1913 650}%
\special{pa 1919 642}%
\special{fp}%
\special{pa 1928 631}%
\special{pa 1935 624}%
\special{fp}%
\special{pa 1947 615}%
\special{pa 1957 610}%
\special{fp}%
\special{pa 1971 604}%
\special{pa 1983 601}%
\special{fp}%
\special{pa 1999 600}%
\special{pa 2011 601}%
\special{fp}%
\special{pa 2026 604}%
\special{pa 2037 607}%
\special{fp}%
\special{pa 2050 614}%
\special{pa 2060 620}%
\special{fp}%
\special{pa 2070 629}%
\special{pa 2077 636}%
\special{fp}%
\special{pa 2085 648}%
\special{pa 2091 658}%
\special{fp}%
\special{pa 2096 673}%
\special{pa 2099 684}%
\special{fp}%
\put(8.0000,-9.0000){\makebox(0,0){$1$}}%
\put(10.0000,-9.0000){\makebox(0,0){$2$}}%
\put(12.0000,-9.0000){\makebox(0,0){$3$}}%
\put(16.0000,-9.0000){\makebox(0,0){$4$}}%
\put(18.0000,-9.0000){\makebox(0,0){$5$}}%
\put(20.0000,-9.0000){\makebox(0,0){$6$}}%
\put(8.0000,-3.0000){\makebox(0,0){$7$}}%
\put(14.0000,-3.0000){\makebox(0,0){$8$}}%
\end{picture}}%
\caption{An example of graphs; the i-vertices are those labeled by $1,\dots,6$ and the f-vertices are those labeled by $7,8$, and there is a loop at the i-vertex labeled by $6$}
\label{fig:graph}
\end{figure}
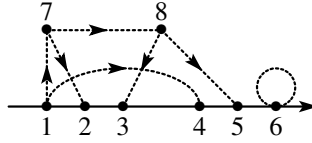
The interval vertices (or i-vertices for short) are placed on the oriented line while the free vertices (or f-vertices for short) are  not on the line.
The i-vertices and the f-vertices of a graph $X$ are labeled by respectively the numbers $1,\dots,\vi$ and $\vi+1,\dots,\vi+\vf$, where $\vi$ and $\vf$ are respectively the numbers of the i-vertices and the f-vertices of $X$, so that the labels of the i-vertices respect the orientation of the real line.
We allow graphs to have a \emph{loop}, an edge that has exactly one i-vertex as its endpoint (see Figure~\ref{fig:graph}).

For a graph $X$, let $E_X$ be the configuration space
\begin{equation}
 E_X\coloneqq\{(f,(y_1,\dots,y_{\vi+\vf}))\in\K\times\Conf_{\vi+\vf}(\R^3)\mid y_i=f(x_i) \text{ for some }x_i\in\R^1\text{ for }i=1,\dots,\vi\},
\end{equation}
where
\begin{equation}
 \Conf_k(M)\coloneqq\{(x_1,\dots,x_k)\in M^{\times k}\mid x_i\ne x_j\text{ if }i\ne j\}
\end{equation}
is the space of $k$-point configurations on a space $M$.

To an oriented edge $\alpha$ of $X$ from the $i$-th vertex to the $j$-th vertex ($i\ne j$), we assign a map
\begin{equation}
 \varphi_{\alpha}\colon E_X\to S^2,\quad
 \varphi_{\alpha}(f,(y_1,\dots,y_{v_i+v_f}))\coloneqq\frac{y_j-y_i}{\abs{y_j-y_i}}.
\end{equation}
To a loop $\alpha$ at $k$-th i-vertex ($1\le i\le\vi$) we assign
\begin{equation}
 \varphi_{\alpha}\colon E_X\to S^2,\quad
 \varphi_{\alpha}(f,(y_1,\dots,y_{v_i+v_f}))\coloneqq\frac{f'(x_k)}{\abs{f'(x_k)}},
\end{equation}
where $x_k\in\R^1$ satisfies $y_k=f(x_k)$.

Let $\vol\in\Omega_{\DR}^2(S^2)$ be a unit volume form of $S^2$ that is anti-symmetric, meaning that $i^*\vol=-\vol$ for the antipodal map $i\colon S^2\to S^2$.
Define $\omega_X\in\Omega^{2e}_{\DR}(E_X)$ by
\begin{equation}
 \omega_X\coloneqq\bigwedge_{\text{edges }\alpha\text{ of }X}\varphi^*_{\alpha}(\vol),
\end{equation}
where $e$ is the number of edges of $X$.
The order of the edges is not important because $\deg\vol=2$ is even.

Let $\pi_X\colon E_X\to\K$ be the first projection.
This is a fiber bundle with fiber
\begin{equation}\label{eq:fibers}
 \pi_X^{-1}(f)=\{y\in\Conf_{\vi+\vf}(\R^3)\mid y_i=f(x_i) \text{ for some }x_i\in\R^1\text{ for }i=1,\dots,\vi\}
\end{equation}
of dimension $\vi+3\vf$.
Integrating $\omega_X$ along the fiber, we get
\begin{equation}\label{eq:I(X)}
 I(X)\coloneqq\pi_{X*}(\omega_X)\in\Omega^{2e-\vi-3\vf}_{\DR}(\K).
\end{equation}

\begin{rem}
The integration \eqref{eq:I(X)} converges since we can compactify all the fibers of $\pi_X$ by adding the boundary faces to \eqref{eq:fibers} so that the maps $\varphi_{\alpha}$ are smoothly extended to the compactification.
See \cite{AxelrodSinger94,BottTaubes94,CCL02,Kohno94}.
\end{rem}

\begin{ex}\label{ex:I(X)_example}
Let $X$ be the graph that has only one edge $\alpha$ joining two i-vertices (Figure~\ref{fig:I(X)_example}, the left).
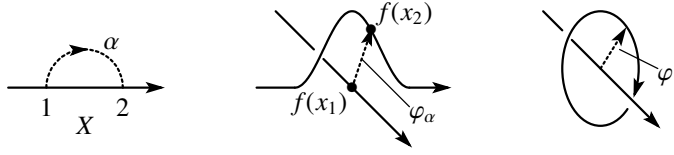
\begin{figure}
\centering
{\unitlength 0.1in%
\begin{picture}(34.0000,7.7000)(2.0000,-8.0000)%
%
\special{pn 13}%
\special{pa 200 500}%
\special{pa 1000 500}%
\special{fp}%
\special{sh 1}%
\special{pa 1000 500}%
\special{pa 933 480}%
\special{pa 947 500}%
\special{pa 933 520}%
\special{pa 1000 500}%
\special{fp}%
%
\special{pn 13}%
\special{pn 13}%
\special{pa 600 300}%
\special{pa 613 300}%
\special{fp}%
\special{pa 630 302}%
\special{pa 643 305}%
\special{fp}%
\special{pa 659 309}%
\special{pa 672 313}%
\special{fp}%
\special{pa 687 320}%
\special{pa 699 326}%
\special{fp}%
\special{pa 713 335}%
\special{pa 724 343}%
\special{fp}%
\special{pa 737 354}%
\special{pa 746 363}%
\special{fp}%
\special{pa 757 376}%
\special{pa 765 386}%
\special{fp}%
\special{pa 774 401}%
\special{pa 780 413}%
\special{fp}%
\special{pa 786 429}%
\special{pa 791 441}%
\special{fp}%
\special{pa 795 457}%
\special{pa 798 470}%
\special{fp}%
\special{pa 799 487}%
\special{pa 800 500}%
\special{fp}%
\put(7.0000,-3.0000){\makebox(0,0)[lb]{$\alpha$}}%
\put(6.0000,-7.0000){\makebox(0,0){$X$}}%
%
\special{pn 13}%
\special{pa 1500 500}%
\special{pa 1700 500}%
\special{fp}%
%
\special{pn 13}%
\special{pa 1850 350}%
\special{pa 2300 800}%
\special{fp}%
\special{sh 1}%
\special{pa 2300 800}%
\special{pa 2267 739}%
\special{pa 2262 762}%
\special{pa 2239 767}%
\special{pa 2300 800}%
\special{fp}%
%
\special{pn 13}%
\special{pa 1800 300}%
\special{pa 1600 100}%
\special{fp}%
%
\special{pn 13}%
\special{pa 2000 500}%
\special{pa 2100 200}%
\special{dt 0.030}%
\special{sh 1}%
\special{pa 2100 200}%
\special{pa 2060 257}%
\special{pa 2083 251}%
\special{pa 2098 270}%
\special{pa 2100 200}%
\special{fp}%
\put(19.7500,-5.2500){\makebox(0,0)[rt]{$f(x_1)$}}%
\put(21.2500,-1.7500){\makebox(0,0)[lb]{$f(x_2)$}}%
%
\special{pn 13}%
\special{ar 3300 400 200 300 0.7853982 0.4636476}%
%
\special{pn 13}%
\special{pa 3484 518}%
\special{pa 3479 534}%
\special{fp}%
\special{sh 1}%
\special{pa 3479 534}%
\special{pa 3518 476}%
\special{pa 3495 483}%
\special{pa 3480 464}%
\special{pa 3479 534}%
\special{fp}%
%
\special{pn 13}%
\special{pa 3150 250}%
\special{pa 3600 700}%
\special{fp}%
\special{sh 1}%
\special{pa 3600 700}%
\special{pa 3567 639}%
\special{pa 3562 662}%
\special{pa 3539 667}%
\special{pa 3600 700}%
\special{fp}%
%
\special{pn 13}%
\special{pa 3100 200}%
\special{pa 3000 100}%
\special{fp}%
%
\special{pn 13}%
\special{pa 2300 500}%
\special{pa 2500 500}%
\special{fp}%
\special{sh 1}%
\special{pa 2500 500}%
\special{pa 2433 480}%
\special{pa 2447 500}%
\special{pa 2433 520}%
\special{pa 2500 500}%
\special{fp}%
%
\special{pn 13}%
\special{pa 3300 400}%
\special{pa 3425 200}%
\special{dt 0.030}%
\special{sh 1}%
\special{pa 3425 200}%
\special{pa 3373 246}%
\special{pa 3397 245}%
\special{pa 3407 267}%
\special{pa 3425 200}%
\special{fp}%
\put(36.0000,-4.0000){\makebox(0,0)[lt]{$\varphi$}}%
%
\special{pn 8}%
\special{pa 3600 400}%
\special{pa 3400 300}%
\special{fp}%
\put(23.0000,-6.0000){\makebox(0,0)[lt]{$\varphi_{\alpha}$}}%
\special{pn 13}%
\special{pa 1700 500}%
\special{pa 1705 500}%
\special{pa 1715 498}%
\special{pa 1720 496}%
\special{pa 1735 487}%
\special{pa 1740 483}%
\special{pa 1755 468}%
\special{pa 1760 462}%
\special{pa 1765 455}%
\special{pa 1770 449}%
\special{pa 1775 441}%
\special{pa 1780 434}%
\special{pa 1790 418}%
\special{pa 1805 391}%
\special{pa 1810 381}%
\special{pa 1815 372}%
\special{pa 1830 342}%
\special{pa 1835 331}%
\special{pa 1840 321}%
\special{pa 1845 310}%
\special{pa 1855 290}%
\special{pa 1860 279}%
\special{pa 1865 269}%
\special{pa 1870 258}%
\special{pa 1885 228}%
\special{pa 1890 219}%
\special{pa 1895 209}%
\special{pa 1910 182}%
\special{pa 1920 166}%
\special{pa 1925 159}%
\special{pa 1930 151}%
\special{pa 1935 145}%
\special{pa 1940 138}%
\special{pa 1945 132}%
\special{pa 1960 117}%
\special{pa 1965 113}%
\special{pa 1980 104}%
\special{pa 1985 102}%
\special{pa 1995 100}%
\special{pa 2005 100}%
\special{pa 2015 102}%
\special{pa 2020 104}%
\special{pa 2035 113}%
\special{pa 2040 117}%
\special{pa 2055 132}%
\special{pa 2060 138}%
\special{pa 2065 145}%
\special{pa 2070 151}%
\special{pa 2075 159}%
\special{pa 2080 166}%
\special{pa 2090 182}%
\special{pa 2105 209}%
\special{pa 2110 219}%
\special{pa 2115 228}%
\special{pa 2130 258}%
\special{pa 2135 269}%
\special{pa 2140 279}%
\special{pa 2145 290}%
\special{pa 2155 310}%
\special{pa 2160 321}%
\special{pa 2165 331}%
\special{pa 2170 342}%
\special{pa 2185 372}%
\special{pa 2190 381}%
\special{pa 2195 391}%
\special{pa 2210 418}%
\special{pa 2220 434}%
\special{pa 2225 441}%
\special{pa 2230 449}%
\special{pa 2235 455}%
\special{pa 2240 462}%
\special{pa 2245 468}%
\special{pa 2260 483}%
\special{pa 2265 487}%
\special{pa 2280 496}%
\special{pa 2285 498}%
\special{pa 2295 500}%
\special{pa 2300 500}%
\special{fp}%
\put(20.0000,-5.0000){\makebox(0,0){$\bullet$}}%
\put(21.0000,-2.0000){\makebox(0,0){$\bullet$}}%
%
\special{pn 8}%
\special{pa 2300 600}%
\special{pa 2050 350}%
\special{fp}%
%
\special{pn 13}%
\special{pn 13}%
\special{pa 400 500}%
\special{pa 400 486}%
\special{fp}%
\special{pa 402 469}%
\special{pa 405 455}%
\special{fp}%
\special{pa 410 438}%
\special{pa 415 425}%
\special{fp}%
\special{pa 422 409}%
\special{pa 428 397}%
\special{fp}%
\special{pa 439 382}%
\special{pa 447 372}%
\special{fp}%
\special{pa 459 359}%
\special{pa 469 349}%
\special{fp}%
\special{pa 482 338}%
\special{pa 494 331}%
\special{fp}%
\special{pa 509 322}%
\special{pa 521 316}%
\special{fp}%
\special{pa 538 310}%
\special{pa 551 306}%
\special{fp}%
\special{pa 569 302}%
\special{pa 582 301}%
\special{fp}%
\special{pa 600 300}%
\special{pa 600 300}%
\special{fp}%
%
\special{pn 13}%
\special{pa 585 301}%
\special{pa 600 300}%
\special{dt 0.030}%
\special{sh 1}%
\special{pa 600 300}%
\special{pa 532 284}%
\special{pa 547 304}%
\special{pa 535 324}%
\special{pa 600 300}%
\special{fp}%
\put(4.0000,-6.0000){\makebox(0,0){$1$}}%
\put(8.0000,-6.0000){\makebox(0,0){$2$}}%
\end{picture}}%
\caption{The graph $X$ in Example~\ref{ex:I(X)_example} (the left), configurations where the image of $\varphi_{\alpha}$ is contained in $\mathrm{supp}(\vol)$ (the center), ``Hopf link'' (the right)}
\label{fig:I(X)_example}
\end{figure}
Then $E_X\approx\K\times\Conf_2(\R^1)$ and $I(X)\in\Omega^0_{\DR}(\K)$ is a function on $\K$.

In this paper we use an anti-symmetric unit volume form $\vol$ whose support is contained in (small) neighborhoods $U_{\pm}$ of the poles $(0,0,\pm 1)\in S^2$.
Suppose $f\in\K$ is ``almost planer,'' meaning that
\begin{itemize}
\item
	the image of $f$ coincides with a knot diagram $D$ on $\R^2\times\{0\}$ except for neighborhoods of crossings of $D$,
\item
	near the crossings the image of $f$ is contained in $\R^2\times(-\epsilon,\epsilon)$, and
\item
	the unit tangent vectors $f'(x)/\abs{f'(x)}$ are not contained in $U_{\pm}$.
\end{itemize}
Then $\varphi_{\alpha}\colon\{f\}\times\Conf_2(\R^1)\to S^2$ has its image in $U_{\pm}$ only on the subspace of $(x_1,x_2)$ such that $f(x_1)$ and $f(x_2)$ are on the over- and under-arcs of a crossing of $D$, one on each arc (Figure~\ref{fig:I(X)_example}, the center).
Each crossing contributes to the value $I(X)(f)$ by the half of its sign; because this contribution is the half of the linking number of the ``Hopf link'' (Figure~\ref{fig:I(X)_example}, the right), which is equal to the sign of the crossing.
\end{ex}

By the generalized Stokes' theorem for fiber integrations, we have
\begin{equation}
 dI(X)=\pi_{X*}(d\omega_X)\pm\pi_{X*}^{\partial}(\omega)=\pm\pi_{X*}^{\partial}(\omega),
\end{equation}
where $\pi_X^{\partial}$ is the restriction of $\pi_X$ to the fiberwise boundary.
There exists ``almost'' 1-1 correspondence between
\begin{itemize}
\item
	the codimension 1 faces of the boundary that nontrivially contribute to $dI(X)$, and
\item
	the graphs obtained from $X$ by contracting one of its edges and \emph{arcs} (segments in $\R^1$ interposed between two i-vertices).
\end{itemize}
Here we in fact need the anti-symmetry of $\vol$.
We thus have
\begin{equation}
 dI(X)=I(\partial X)+\text{(correction terms)},
\end{equation}
where $\partial X$ is a formal sum of graphs obtained from $X$ by contracting one of its edges and arcs.
The above correspondence is not rigorously 1-1 and we need ``correction terms,'' that are conjectured to vanish.
We can therefore get a closed form of $\K$ if we have a \emph{graph cocycle}, a formal sum $X$ of graphs with $\partial X=0$ (and if we have appropriate correction terms).
It is known that any $\R$-valued Vassiliev invariant can be produced from a \emph{trivalent} graph cocycle.

In \cite{K07,K09} the second named author has given an example of \emph{non-trivalent} graph cocycle
\begin{equation}
 X=\sum_{1\le k\le 9}a_kX_k,\quad
 (a_1,\dots,a_9)=(-2,1,2,-2,2,-1,1,-1,1) 
\end{equation}
(see Figure~\ref{fig:Nontrivalent_graph}), and has proved that $I(X)\in H^1_{\DR}(\K)$ is not zero.
\begin{figure}
\centering
\input{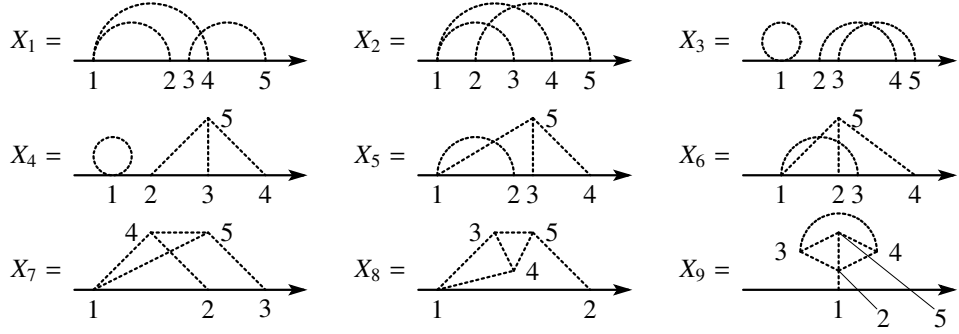}
\caption{The graphs $X_1,\dots,X_9$ that give a graph cocycle $\sum_ia_iX_i$; the edges are oriented from the vertex with the smaller labels}
\label{fig:Nontrivalent_graph}
\end{figure}
This follows from

\begin{thm}[{\cite{K09}}]\label{thm:Gramain}
The differential form $I(X)\in\Omega^1_{\DR}(\K)$ is closed, and its integration over the \emph{Gramain cycle} $G_f$ (see Remark~\ref{rem:Gramain} below) is equal to the \emph{Casson invariant} $v_2(f)$.
\end{thm}

\begin{rem}\label{rem:Gramain}
The \emph{Gramain 1-cycle} $G_f\colon S^1\to\K$ for $f\in\K$ is a cycle that rotates $f$ around the ``long axis'' $\R^1\times\{(0,0)\}$.
Explicitly $G_f$ is given by
\begin{equation}
 G_f(\theta)(x)\coloneqq
 \begin{pmatrix}
  1 &            & \\
    & \cos\theta & \\
    &            & \sin\theta
 \end{pmatrix}
 f(x)
 \quad\text{for}\quad
 \theta\in S^1,\ x\in\R^1.
\end{equation}
Mortier \cite[Theorem~4.1]{Mortier15} has given a $1$-cocycle $\alpha_3^1$ of $\K$ in a combinatorial way and has proved that
\begin{equation}
 \ang{\alpha_3^1,G_f}=v_2(f)
 \quad\text{and}\quad
 \ang{\alpha_3^1,p_*\FH_{(f,w)}}=6v_3(f)-w\cdot v_2(f)
\end{equation}
for $(f,w)\in\K\times\Z\simeq\fK$.
This result motivates us to compute the integration of $I(X)$ over the FH-cycles.
\end{rem}

\section{Integration of $I(X)$ over the Fox-Hatcher cycle}
Recall that $p\colon\fK\to\K$ is the map forgetting the framing of $\wtf$.
For any $\wtf\in\fK$ we define
\begin{equation}\label{eq:def_v}
 v(\wtf)\coloneqq\int_{p_*\FH_{\wtf}}I(X)=\sum_{1\le k\le 9}a_k\int_{p_*\FH_{\wtf}}I(X_k).
\end{equation}
This gives an isotopy invariant $v$ for framed long knots.
Our goal is to describe $v$ as a linear combination of the Vassiliev invariants of order less or equal to three.

\subsection{The invariant $v$ is of order three}
For any $\wtf\in\fK$ and crossings $c_1,\dots,c_n$ of its diagram, define
\begin{equation}\label{eq:derivative}
 D^nv(\wtf)\coloneqq\sum_{\epsilon_1,\dots,\epsilon_n\in\{+1,-1\}}\epsilon_1\dotsb\epsilon_nv(\wtf{}_{\epsilon_1,\dots,\epsilon_n}),
\end{equation}
where $\wtf_{\epsilon_1,\dots,\epsilon_n}$ is a framed long knot obtained by changing, if necessary, the crossings $c_i$ so that its sign is equal to $\epsilon_i$.
What we want to show is $D^4v(\wtf)=0$ for any choice of $\wtf$ and $c_1,\dots,c_4$.

Let $c_1,c_2,c_3$ be (part of the) crossings of a diagram $D$ of $\wtf\in\fK$ respecting the Gauss diagram $G$ (Definition~\ref{def:crossings_respecting_Gauss_diagram}).
Let us perform the FH-moves (described in \S\ref{s:FH_cycle}) on all the crossings $c$ of $D$ and the corresponding newborn crossings $c'$.
The Gauss diagram that the three crossings under consideration respect changes as in Figure~\ref{fig:Gauss_FH} when the FH-move is performed on one of $c_i$ and $c_i'$ ($i=1,2,3$), and in the sequence of the FH-moves realizing the FH-cycle, six Gauss diagrams (some of which may be equal to each other) respected by the three crossings under consideration form a cycle. 
Figures~\ref{fig:Gauss_cycle_1}, \ref{fig:Gauss_cycle_2} and \ref{fig:Gauss_cycle_3} show three such cycles.
\begin{figure}
\centering
\input{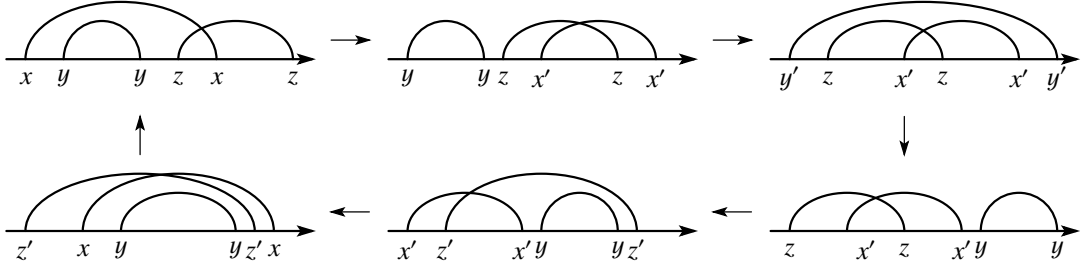}
\caption{Type I cycle of the Gauss diagrams respecting three crossings under consideration; $\{x,y,z\}=\{c_1,c_2,c_3\}$}
\label{fig:Gauss_cycle_1}
\end{figure}
\begin{figure}
\centering
\input{Gauss_cycle_2}
\caption{Type II cycle of the Gauss diagrams respecting three crossings under consideration; $\{x,y,z\}=\{c_1,c_2,c_3\}$}
\label{fig:Gauss_cycle_2}
\end{figure}
\begin{figure}
\centering
\input{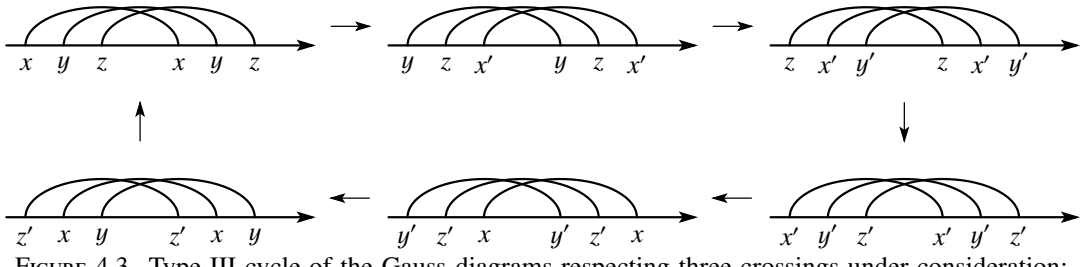}
\caption{Type III cycle of the Gauss diagrams respecting three crossings under consideration; $\{x,y,z\}=\{c_1,c_2,c_3\}$}
\label{fig:Gauss_cycle_3}
\end{figure}
There are 15 Gauss diagrams with three edges, and only 10 of them are included in these three cycles.
The remaining five Gauss diagrams form the other two cycles, that we omit since in fact they do not contribute to our computation in \S\ref{ss:D^3v}.

\begin{thm}\label{thm:D^3v}
$D^3v(\wtf)$ is given by
\begin{equation}
 D^3v(\wtf)=
 \begin{cases}
  -2 & \text{if }c_1,c_2\text{ and }c_3\text{ respect one of the Gauss diagrams in type I cycle},\\
  2  & \text{if }c_1,c_2\text{ and }c_3\text{ respect one of the Gauss diagrams in type II cycle},\\
  6  & \text{if }c_1,c_2\text{ and }c_3\text{ respect the unique Gauss diagram in type III cycle},\\
  0  & \text{otherwise}.
 \end{cases}
\end{equation}
\end{thm}

\begin{cor}\label{cor:v_order_3}
The invariant $v$ is a Vassiliev invariant for framed long knots of order exactly three.
\end{cor}
\begin{proof}
Let $c_1,\dots,c_4$ be crossings of a diagram of $\wtf\in\fK$.
Let $\wtf_{\pm}$ be knots obtained by changing $c_4$ so that its sign is respectively $\pm 1$.
Then by definition
\begin{equation}
 D^4v(f)=D^3v(f_+)-D^3v(f_-).
\end{equation}
Moreover $c_1,c_2$ and $c_3$ of $f_+$ and $f_-$ respect the same Gauss diagram.
Thus we have $D^3v(\wtf_+)=D^3v(\wtf_-)$ by Theorem~\ref{thm:D^3v}, concluding $D^4v(\wtf)=0$.

Theorem~\ref{thm:D^3v} also says that $D^3v(\wtf)$ can be nonzero, and $v$ is not of order two nor less.
\end{proof}

The next subsection is devoted to the proof of Theorem~\ref{thm:D^3v}.

\subsection{Computation of $D^3v$}\label{ss:D^3v}
As in Example~\ref{ex:I(X)_example}, we assume that
\begin{itemize}
\item
	$\vol\in\Omega^2_{\DR}(S^2)$ is an anti-symmetric unit volume form of $S^2$ whose support is contained in small neighborhoods of poles $(0,0,\pm 1)\in S^2$, and
\item
	we compute $D^3v(\wtf)$ after transforming $\wtf$ to be ``almost planar.''
\end{itemize}
For $k=1,\dots,9$, consider the pullback square
\begin{equation}
\begin{split}
 \xymatrix{
 (p\circ\FH_{\wtf})^*E_{X_k}\ar@{->}[rr]^-{\overline{p\circ\FH_{\wtf}}}\ar@{->}[d]_-{\pi'_{X_k}} & & E_{X_k}\ar@{->}[d]^-{\pi_{X_k}}\\
 S^1\ar@{->}[r]^-{\FH_{\wtf}} & \fK\ar@{->}[r]^-p & \K
 }
\end{split}
\end{equation}
Then
\begin{equation}\label{eq:integral_pullback}
 \int_{p_*\FH_{\wtf}}I(X_k)
 =\int_{S^1}(p\circ\FH_{\wtf})^*\pi_{X_k*}\omega_{X_k}
 =\int_{S^1}\pi'_{X_k*}\overline{p\circ\FH_{\wtf}}{}^*\omega_{X_k}
 =\int_{(p\circ\FH_{\wtf})^*E_{X_k}}\overline{p\circ\FH_{\wtf}}{}^*\omega_{X_k}.
\end{equation}
Note that $(p\circ\FH_{\wtf})^*E_{X_k}$ is explicitly given by
\begin{equation}
 (p\circ\FH_{\wtf})^*E_{X_k}
 =\left\{
  (p(\FH_{\wtf}(\theta)),y)\in\K\times\Conf_{\vi+\vf}(\R^3)\mathrel{}\middle|\mathrel{}
 \begin{array}{l}
 \theta\in S^1,\ y_i=p(\FH_{\wtf}(\theta))(x_i)\\
 \text{for some }x_i\in\R^1,\ 1\le i\le\vi
 \end{array}
 \right\}\subset E_{X_k}.
\end{equation}
Suppose a diagram $D$ of $\wtf$ has $n$ crossings.
Then $\FH_{\wtf}$ can be realized on knot diagram by the sequence of $2n$ FH-moves on $c$ or $c'$, where $c$ is one of the crossings of $D$ and $c'$ is a newly created crossing after the FH-move on $c$.
We can decompose $S^1$ into $2n$ intervals
\begin{equation}\label{eq:decomposition_S^1}
 S^1=\bigcup_c(I_c\cup I_{c'})
\end{equation}
such that $\FH_{\wtf}$ restricted on ${I_c}$ (resp.~$I_{c'}$) realizes the FH-move on $c$ (resp.~$c'$).

\begin{defn}
Under the above setting, define
\begin{equation}
 E_{k;c,c'}\coloneqq\{(p_*(\FH_{\wtf}(\theta)),y)\in(p\circ H_{\wtf})^*E_{X_k}\mid \theta\in I_c\cup I_{c'}\}.
\end{equation}
\end{defn}

By definition we have
\begin{equation}\label{eq:E_X_k_decomposes_into_E_c}
 (p\circ\FH_{\wtf})^*E_{X_k}=\bigcup_cE_{k;c,c'}
 \quad\text{and hence}\quad
 \int_{(p\circ\FH_{\wtf})^*E_X}\omega_{X_k}=\sum_c\int_{E_{k;c,c'}}\overline{p\circ\FH_{\wtf}}{}^*\omega_{X_k}.
\end{equation}
Combining
\eqref{eq:def_v},
\eqref{eq:derivative},
\eqref{eq:integral_pullback} and
\eqref{eq:E_X_k_decomposes_into_E_c},
we have
\begin{equation}
 D^3v(\wtf)=\sum_{1\le k\le 9}a_k\sum_c\sum_{\epsilon_1,\epsilon_2,\epsilon_3\in\{+1,-1\}}\epsilon_1\epsilon_2\epsilon_3\int_{E_{k;c,c'}}\overline{p\circ\FH_{\wtf{}_{\epsilon_1,\epsilon_2,\epsilon_3}}}{}^*\omega_{X_k}.
\end{equation}


\subsubsection{Eliminating $X_3,\dots,X_9$}
Let $h_i$ ($i=1,2,3$) be the ``distance'' between two arcs at $c_i$, $i=1,2,3$ (Figure~\ref{fig:h_i}).
\begin{figure}
\centering
{\unitlength 0.1in%
\begin{picture}(10.0000,6.3000)(3.0000,-8.0000)%
\special{pn 13}%
\special{pa 500 600}%
\special{pa 505 600}%
\special{pa 515 598}%
\special{pa 520 596}%
\special{pa 535 587}%
\special{pa 540 583}%
\special{pa 555 568}%
\special{pa 560 562}%
\special{pa 565 555}%
\special{pa 570 549}%
\special{pa 575 541}%
\special{pa 580 534}%
\special{pa 590 518}%
\special{pa 605 491}%
\special{pa 610 481}%
\special{pa 615 472}%
\special{pa 630 442}%
\special{pa 635 431}%
\special{pa 640 421}%
\special{pa 645 410}%
\special{pa 655 390}%
\special{pa 660 379}%
\special{pa 665 369}%
\special{pa 670 358}%
\special{pa 685 328}%
\special{pa 690 319}%
\special{pa 695 309}%
\special{pa 710 282}%
\special{pa 720 266}%
\special{pa 725 259}%
\special{pa 730 251}%
\special{pa 735 245}%
\special{pa 740 238}%
\special{pa 745 232}%
\special{pa 760 217}%
\special{pa 765 213}%
\special{pa 780 204}%
\special{pa 785 202}%
\special{pa 795 200}%
\special{pa 805 200}%
\special{pa 815 202}%
\special{pa 820 204}%
\special{pa 835 213}%
\special{pa 840 217}%
\special{pa 855 232}%
\special{pa 860 238}%
\special{pa 865 245}%
\special{pa 870 251}%
\special{pa 875 259}%
\special{pa 880 266}%
\special{pa 890 282}%
\special{pa 905 309}%
\special{pa 910 319}%
\special{pa 915 328}%
\special{pa 930 358}%
\special{pa 935 369}%
\special{pa 940 379}%
\special{pa 945 390}%
\special{pa 955 410}%
\special{pa 960 421}%
\special{pa 965 431}%
\special{pa 970 442}%
\special{pa 985 472}%
\special{pa 990 481}%
\special{pa 995 491}%
\special{pa 1010 518}%
\special{pa 1020 534}%
\special{pa 1025 541}%
\special{pa 1030 549}%
\special{pa 1035 555}%
\special{pa 1040 562}%
\special{pa 1045 568}%
\special{pa 1060 583}%
\special{pa 1065 587}%
\special{pa 1080 596}%
\special{pa 1085 598}%
\special{pa 1095 600}%
\special{pa 1100 600}%
\special{fp}%
%
\special{pn 13}%
\special{pa 300 600}%
\special{pa 500 600}%
\special{fp}%
%
\special{pn 13}%
\special{pa 1100 600}%
\special{pa 1300 600}%
\special{fp}%
%
\special{pn 13}%
\special{pa 650 450}%
\special{pa 1000 800}%
\special{fp}%
\special{sh 1}%
\special{pa 1000 800}%
\special{pa 967 739}%
\special{pa 962 762}%
\special{pa 939 767}%
\special{pa 1000 800}%
\special{fp}%
%
\special{pn 13}%
\special{pa 400 200}%
\special{pa 600 400}%
\special{fp}%
%
\special{pn 8}%
\special{pa 800 600}%
\special{pa 800 200}%
\special{fp}%
\special{sh 1}%
\special{pa 800 200}%
\special{pa 780 267}%
\special{pa 800 253}%
\special{pa 820 267}%
\special{pa 800 200}%
\special{fp}%
%
\special{pn 8}%
\special{pa 800 200}%
\special{pa 800 600}%
\special{fp}%
\special{sh 1}%
\special{pa 800 600}%
\special{pa 820 533}%
\special{pa 800 547}%
\special{pa 780 533}%
\special{pa 800 600}%
\special{fp}%
\put(11.0000,-3.0000){\makebox(0,0)[lb]{$h_i$}}%
%
\special{pn 8}%
\special{pa 1100 300}%
\special{pa 800 400}%
\special{dt 0.030}%
\put(7.7500,-6.2500){\makebox(0,0)[rt]{$c_i$}}%
\end{picture}}%
\caption{$h_i$ ($i=1,2,3$)}
\label{fig:h_i}
\end{figure}
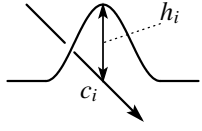
We may compute $D^3v(\wtf)$ in the limit $h_i\to 0$ ($i=1,2,3$) since $v$ is an invariant.
In this limit, only the graphs $X_1$ and $X_2$ essentially contribute to $D^3v(\wtf)$;

\begin{prop}\label{prop:cancellation}
\begin{enumerate}[(1)]
\item
	For $k=1,\dots,9$ and any crossing $c$ other than $c_1,c_2,c_3$, we have
	\begin{equation}\label{eq:cancellation}
	 \lim_{h_1,h_2,h_3\to 0}\sum_{\epsilon_1,\epsilon_2,\epsilon_3\in\{+1,-1\}}\epsilon_1\epsilon_2\epsilon_3\int_{E_{k;c,c'}}\overline{p\circ\FH_{\wtf{}_{\epsilon_1,\epsilon_2,\epsilon_3}}}{}^*\omega_{X_k}=0.
	\end{equation}
\item
	If $k=3,\dots,9$, then the equation \eqref{eq:cancellation} also holds for $c\in\{c_1,c_2,c_3\}$.
\end{enumerate}
Consequently
\begin{equation}\label{eq:X_1_and_X_2_determine_D^3v}
 D^3v(\wtf)=\lim_{h_1,h_2,h_3\to 0}\sum_{k=1,2}a_k\sum_{c\in\{c_1,c_2,c_3\}}\sum_{\epsilon_1,\epsilon_2,\epsilon_3\in\{+1,-1\}}\epsilon_1\epsilon_2\epsilon_3\int_{E_{k;c,c'}}\overline{p\circ\FH_{\wtf{}_{\epsilon_1,\epsilon_2,\epsilon_3}}}{}^*\omega_{X_k}.
\end{equation}
\end{prop}

\begin{proof}[{Proof of Proposition~\ref{prop:cancellation} (1)}]
Let $-1<p_i<q_i<1$ ($i=1,2,3$) with $p_1<p_2<p_3$ be the real numbers such that $f(p_i)$ and $f(q_i)$ correspond to $c_i$, and let $A_i$, $B_i$ be small open intervals that include respectively $p_i$ and $q_i$ (see Figure~\ref{fig:cancellation_1}).
Let $E_{k;c,c',A_1}\subset E_{k;c,c'}$ be the subspace consisting of $(\theta,y)$ with no $y_j$ ($1\le j\le\vi$) being in $A_1$.
\begin{figure}
\centering
\input{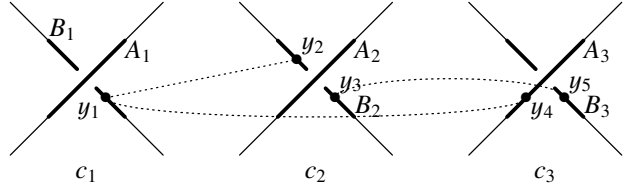}
\caption{An element of $E_{1;c_1,c_1',A_1}$}
\label{fig:cancellation_1}
\end{figure}
Then even if we set $h_1=0$, any two points $y_j$ and $y_{j'}$ corresponding to endpoints of a single edge of $X_k$ do not collide in $E_{k;c,c',A_1}$, and the maps $\varphi_{\alpha}$ and hence the integrand $\omega_{X_k}$ can be defined on $E_{k;c,c',A_1}$.
This implies
\begin{equation}\label{eq:limit_is_zero_1}
\lim_{h_1\to 0}\left(\int_{E_{k;c,c',A_1}}\overline{p\circ\FH_{\wtf_{+1,\epsilon_2,\epsilon_3}}}{}^*\omega_{X_k}-\int_{E_{k;c,c',A_1}}\overline{p\circ\FH_{\wtf_{-1,\epsilon_2,\epsilon_3}}}{}^*\omega_{X_k}\right)=0.
\end{equation}
If we analogously define $E_{k;c,c',A_m}$ and $E_{k;c,c',B_m}$, then similar cancellation to \eqref{eq:limit_is_zero_1} occurs for them.
Moreover we have
\begin{equation}
 \bigcup_{m=1,2,3}(E_{k;c,c',A_m}\cup E_{k;c,c',B_m})=E_{k;c,c'}
\end{equation}
because no $X_k$ has six or more i-vertices.
Although $A_1,\dots,B_3$ are not disjoint, we can arrange them to be disjoint by considering their difference sets and intersections (on which the same argument is valid).
Thus we have \eqref{eq:cancellation}.
\end{proof}

\begin{proof}[{Proof of Proposition~\ref{prop:cancellation} (2) for $k=7,8,9$}]
It is enough to consider the case $c=c_1$; the cases $c=c_2,c_3$ can be proved similarly.

The similar argument in the proof of (1) also implies \eqref{eq:limit_is_zero_1} with $A_1$ and $h_1$ replaced respectively by $A_m$ (or $B_m$) and $h_m$, $m=2,3$.
We thus complete the proof, because $X_k$ ($k=7,8,9$) has three or less i-vertices and we have
\begin{equation}
 E_{k;c_1,c_1'}=\bigcup_{m=2,3}(E_{k;c_1,c'_1,A_m}\cup E_{k;c_1,c'_1,B_m}).\qedhere
\end{equation}
\end{proof}

\begin{proof}[{Proof of Proposition~\ref{prop:cancellation} (2) for $k=5,6$}]
It is enough to consider the case $c=c_1$.

Let $E\subset E_{k;c_1,c_1'}$ be the subspace of $E_{k;c_1,c_1'}$ consisting of $(\theta,y)$ with each of $A_2,B_2,A_3$ and $B_3$ containing at least one $y_j$ corresponding to an i-vertex $j$ of $X_k$.
Then the integrations in \eqref{eq:cancellation} with $E_{k;c_1,c_1'}$ replaced by $E_{k;c_1,c_1'}\setminus E$ are defined even if we set $h_m=0$ for at least one $m\in\{2,3\}$, and the cancellation similar to \eqref{eq:limit_is_zero_1} occurs, similarly as the above proof of (2) for $k=7,8,9$.
Thus it suffices to show \eqref{eq:cancellation} with $E_{k;c_1,c_1'}$ replaced by $E$.
Since $X_k$ ($k=5,6$) has four i-vertices, each of $A_2,B_2,A_3$ and $B_3$ contains exactly one point on $E$.
We divide $E$ into two subspaces;
\begin{itemize}
\item[\emph{Type I};]
	the subspace $E_I$ of $E$ consisting of $(\theta,y)$ with $y_5\in\R^3$ outside neighborhoods of $c_2$ and $c_3$.
	Then two i-vertices ($4$ and $5$ in the case of Figure~\ref{fig:cancellation_2}) corresponding to the points in $A_m\cup B_m$ are not joined by any edge, for at least one $m=2,3$.
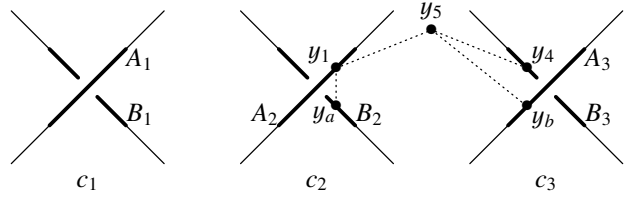
\begin{figure}
\centering
{\unitlength 0.1in%
\begin{picture}(32.0000,9.0000)(4.0000,-10.2700)%
%
\special{pn 8}%
\special{pa 400 1000}%
\special{pa 1200 200}%
\special{fp}%
%
\special{pn 8}%
\special{pa 1600 1000}%
\special{pa 2400 200}%
\special{fp}%
%
\special{pn 8}%
\special{pa 2800 1000}%
\special{pa 3600 200}%
\special{fp}%
%
\special{pn 8}%
\special{pa 400 200}%
\special{pa 750 550}%
\special{fp}%
%
\special{pn 8}%
\special{pa 850 650}%
\special{pa 1200 1000}%
\special{fp}%
%
\special{pn 8}%
\special{pa 2050 650}%
\special{pa 2400 1000}%
\special{fp}%
%
\special{pn 8}%
\special{pa 3250 650}%
\special{pa 3600 1000}%
\special{fp}%
%
\special{pn 8}%
\special{pa 1600 200}%
\special{pa 1950 550}%
\special{fp}%
%
\special{pn 8}%
\special{pa 2800 200}%
\special{pa 3150 550}%
\special{fp}%
\put(8.0000,-11.0000){\makebox(0,0){$c_1$}}%
\put(20.0000,-11.0000){\makebox(0,0){$c_2$}}%
\put(32.0000,-11.0000){\makebox(0,0){$c_3$}}%
%
\special{pn 20}%
\special{pa 600 800}%
\special{pa 1000 400}%
\special{fp}%
%
\special{pn 20}%
\special{pa 1800 800}%
\special{pa 2200 400}%
\special{fp}%
%
\special{pn 20}%
\special{pa 3000 800}%
\special{pa 3400 400}%
\special{fp}%
%
\special{pn 20}%
\special{pa 600 400}%
\special{pa 750 550}%
\special{fp}%
%
\special{pn 20}%
\special{pa 1800 400}%
\special{pa 1950 550}%
\special{fp}%
%
\special{pn 20}%
\special{pa 3000 400}%
\special{pa 3150 550}%
\special{fp}%
%
\special{pn 20}%
\special{pa 850 650}%
\special{pa 1000 800}%
\special{fp}%
%
\special{pn 20}%
\special{pa 2050 650}%
\special{pa 2200 800}%
\special{fp}%
%
\special{pn 20}%
\special{pa 3250 650}%
\special{pa 3400 800}%
\special{fp}%
\put(10.0000,-4.0000){\makebox(0,0)[lt]{$A_1$}}%
\put(10.0000,-8.0000){\makebox(0,0)[lb]{$B_1$}}%
\put(18.0000,-8.0000){\makebox(0,0)[rb]{$A_2$}}%
\put(34.0000,-4.0000){\makebox(0,0)[lt]{$A_3$}}%
\put(22.0000,-8.0000){\makebox(0,0)[lb]{$B_2$}}%
\put(34.0000,-8.0000){\makebox(0,0)[lb]{$B_3$}}%
\put(26.0000,-3.0000){\makebox(0,0){$\bullet$}}%
\put(21.0000,-5.0000){\makebox(0,0){$\bullet$}}%
\put(21.0000,-7.0000){\makebox(0,0){$\bullet$}}%
\put(31.0000,-5.0000){\makebox(0,0){$\bullet$}}%
\put(31.0000,-7.0000){\makebox(0,0){$\bullet$}}%
\put(21.0000,-7.0000){\makebox(0,0)[rt]{$y_a$}}%
\put(20.7500,-4.7500){\makebox(0,0)[rb]{$y_1$}}%
\put(26.0000,-2.0000){\makebox(0,0){$y_5$}}%
\put(31.2500,-7.2500){\makebox(0,0)[lt]{$y_b$}}%
\put(31.2500,-4.7500){\makebox(0,0)[lb]{$y_4$}}%
%
\special{pn 8}%
\special{pa 2600 300}%
\special{pa 3100 700}%
\special{dt 0.030}%
%
\special{pn 8}%
\special{pa 2100 500}%
\special{pa 2600 300}%
\special{dt 0.030}%
%
\special{pn 8}%
\special{pa 2100 500}%
\special{pa 2100 700}%
\special{dt 0.030}%
%
\special{pn 8}%
\special{pa 3100 500}%
\special{pa 2600 300}%
\special{dt 0.030}%
\end{picture}}%
\caption{Proposition~\ref{prop:cancellation} (2) for $k=5,6$, Type I subspace ($m=3$, $\{a,b\}=\{2,3\}$); one of the arcs $A_1$ and $B_1$ moves in the FH-move on $c_1$.}
\label{fig:cancellation_2}
\end{figure}
	Even if we set $h_m=0$ and these two points may collide, all the maps $\varphi_{\alpha}$ and hence $\omega_{X_k}$ are still defined on $E_I$, and the cancellation similar to \eqref{eq:limit_is_zero_1} occurs.
\item[\emph{Type II};]
	the subspace $E_{II}$ of $E$ consisting of $(\theta,y)$ with $y_5\in\R^3$ in a neighborhoods of $c_m$, $m\in\{2,3\}$ (see Figure~\ref{fig:cancellation_3}; setting $h_2=0$ or $h_3=0$ are problematic on this subspace).
\begin{figure}
\centering
{\unitlength 0.1in%
\begin{picture}(32.0000,8.2700)(4.0000,-10.2700)%
%
\special{pn 8}%
\special{pa 400 1000}%
\special{pa 1200 200}%
\special{fp}%
%
\special{pn 8}%
\special{pa 1600 1000}%
\special{pa 2400 200}%
\special{fp}%
%
\special{pn 8}%
\special{pa 2800 1000}%
\special{pa 3600 200}%
\special{fp}%
%
\special{pn 8}%
\special{pa 400 200}%
\special{pa 750 550}%
\special{fp}%
%
\special{pn 8}%
\special{pa 850 650}%
\special{pa 1200 1000}%
\special{fp}%
%
\special{pn 8}%
\special{pa 2050 650}%
\special{pa 2400 1000}%
\special{fp}%
%
\special{pn 8}%
\special{pa 3250 650}%
\special{pa 3600 1000}%
\special{fp}%
%
\special{pn 8}%
\special{pa 1600 200}%
\special{pa 1950 550}%
\special{fp}%
%
\special{pn 8}%
\special{pa 2800 200}%
\special{pa 3150 550}%
\special{fp}%
\put(8.0000,-11.0000){\makebox(0,0){$c_1$}}%
\put(20.0000,-11.0000){\makebox(0,0){$c_2$}}%
\put(32.0000,-11.0000){\makebox(0,0){$c_3$}}%
%
\special{pn 20}%
\special{pa 600 800}%
\special{pa 1000 400}%
\special{fp}%
%
\special{pn 20}%
\special{pa 1800 800}%
\special{pa 2200 400}%
\special{fp}%
%
\special{pn 20}%
\special{pa 3000 800}%
\special{pa 3400 400}%
\special{fp}%
%
\special{pn 20}%
\special{pa 600 400}%
\special{pa 750 550}%
\special{fp}%
%
\special{pn 20}%
\special{pa 1800 400}%
\special{pa 1950 550}%
\special{fp}%
%
\special{pn 20}%
\special{pa 3000 400}%
\special{pa 3150 550}%
\special{fp}%
%
\special{pn 20}%
\special{pa 850 650}%
\special{pa 1000 800}%
\special{fp}%
%
\special{pn 20}%
\special{pa 2050 650}%
\special{pa 2200 800}%
\special{fp}%
%
\special{pn 20}%
\special{pa 3250 650}%
\special{pa 3400 800}%
\special{fp}%
\put(10.0000,-4.0000){\makebox(0,0)[lt]{$A_1$}}%
\put(10.0000,-8.0000){\makebox(0,0)[lb]{$B_1$}}%
\put(18.0000,-8.0000){\makebox(0,0)[rb]{$A_2$}}%
\put(34.0000,-4.0000){\makebox(0,0)[lt]{$A_3$}}%
\put(22.0000,-8.0000){\makebox(0,0)[lb]{$B_2$}}%
\put(34.0000,-8.0000){\makebox(0,0)[lb]{$B_3$}}%
\put(30.0000,-6.0000){\makebox(0,0){$\bullet$}}%
\put(21.0000,-5.0000){\makebox(0,0){$\bullet$}}%
\put(21.0000,-7.0000){\makebox(0,0){$\bullet$}}%
\put(31.0000,-5.0000){\makebox(0,0){$\bullet$}}%
\put(31.0000,-7.0000){\makebox(0,0){$\bullet$}}%
\put(21.0000,-7.0000){\makebox(0,0)[rt]{$y_a$}}%
\put(20.2500,-4.7500){\makebox(0,0)[rb]{$y_1$}}%
\put(27.0000,-4.0000){\makebox(0,0)[rb]{$y_5$}}%
\put(31.2500,-7.2500){\makebox(0,0)[lt]{$y_b$}}%
\put(31.2500,-4.7500){\makebox(0,0)[lb]{$y_4$}}%
%
\special{pn 8}%
\special{pa 2100 500}%
\special{pa 3000 600}%
\special{dt 0.030}%
%
\special{pn 8}%
\special{pa 2100 500}%
\special{pa 2100 700}%
\special{dt 0.030}%
%
\special{pn 8}%
\special{pa 3100 700}%
\special{pa 3000 600}%
\special{dt 0.030}%
%
\special{pn 8}%
\special{pa 3100 500}%
\special{pa 3000 600}%
\special{dt 0.030}%
%
\special{pn 8}%
\special{pa 2700 400}%
\special{pa 2940 560}%
\special{fp}%
\special{sh 1}%
\special{pa 2940 560}%
\special{pa 2896 506}%
\special{pa 2896 530}%
\special{pa 2873 540}%
\special{pa 2940 560}%
\special{fp}%
\end{picture}}%
\caption{Proposition~\ref{prop:cancellation} (2) for $k=5,6$, Type II subspace ($m=3$, $\{a,b\}=\{2,3\}$)}
\label{fig:cancellation_3}
\end{figure}
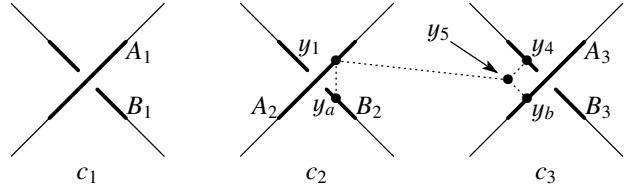
	On $E_{II}$ at least one edge $\alpha$ of $X_k$ joins the vertex $5$ and $j$ with the corresponding point $y_j$ not on $A_m\cup B_m$ ($j=1$ in the case of Figure~\ref{fig:cancellation_3}).
	Then the image of $\varphi_{\alpha}$ is not included in $\supp(\vol)$ and hence $\varphi^*_{\alpha}\vol=0$, because $\supp(\vol)$ is assumed to be in neighborhoods of $(0,0,\pm 1)\in S^2$ and our $\wtf$ is almost planar.
	The integrand $\omega_{X_k}$ is therefore zero on $E_{II}$.\qedhere
\end{itemize}
\end{proof}

\begin{proof}[{Proof of Proposition~\ref{prop:cancellation} (2) for $k=4$}]
Consider the case $c=c_1$ (the same arguments are valid for $c=c_2,c_3$).
Let $E\subset E_{4;c_1,c_1'}$ be the subspace consisting of $(\theta,y)$ with each of $A_2,B_2,A_3$ and $B_3$ contains at least one point.
It is then enough to show \eqref{eq:cancellation} with $E_{4;c_1,c_1'}$ replaced by $E$, as in the above proofs.

As $X_4$ has four i-vertices, each of $A_2,B_2,A_3$ and $B_3$ contains exactly one point on $E$.
In particular $y_1\in A_2$, and the map $\varphi_{\alpha}$ for the loop $\alpha$ at the vertex $1$ has the image outside $\supp(\vol)$ by our assumption on $\wtf$ and $\vol$, and hence $\omega_{X_4}$ vanishes on $E$.
\end{proof}

\begin{proof}[{Proof of Proposition~\ref{prop:cancellation} (2) for $k=3$}]
Again consider the case $c=c_1$.
Let $E\subset E_{3;c_1,c_1'}$ be the subspace  consisting of $(\theta,y)$ satisfying both (i) and (ii);
\begin{enumerate}[(i)]
\item
	$y_1$ is on the arc $C$ that moves in the FH-moves on $c_1$,
\item
	each of $A_2,B_2,A_3$ and $B_3$ contains exactly one of $y_2,\dots,y_5$.
\end{enumerate}
Then it suffices to show \eqref{eq:cancellation} with $E_{3;c_1,c_1'}$ replaced by $E$.
This is because:
\begin{itemize}
\item
	If $E'$ denotes the subspace of $E_{3;c_1,c_1'}$ consisting of $(\theta,y)$ that does not satisfy (ii), then the integrations in \eqref{eq:cancellation} with $E_{3;c_1,c_1'}$ replaced by $E'$ are defined even if we set $h_m=0$ for at least one $m\in\{2,3\}$ and the cancellation similar to \eqref{eq:limit_is_zero_1} occurs, by the same reason as in the above proofs. 
\item
	If $E''$ denotes the subspace of $E_{3;c_1,c_1'}$ consisting of $(\theta,y)$ that satisfies (ii) but does not satisfy (i). then the map $\varphi_{\alpha}$ ($\alpha$ is the loop of $X_3$ at the i-vertex labeled by $1$) has its image outside on $\supp(\vol)$ since $\wtf$ is supposed to be almost planar, and hence $\omega_{X_3}$ vanishes on $E''$.
\end{itemize}
Figure~\ref{fig:I(X_3)} shows the configurations in $E$ that may non-trivially contribute to the integration of $I(X_3)$.
\begin{figure}
\centering
\input{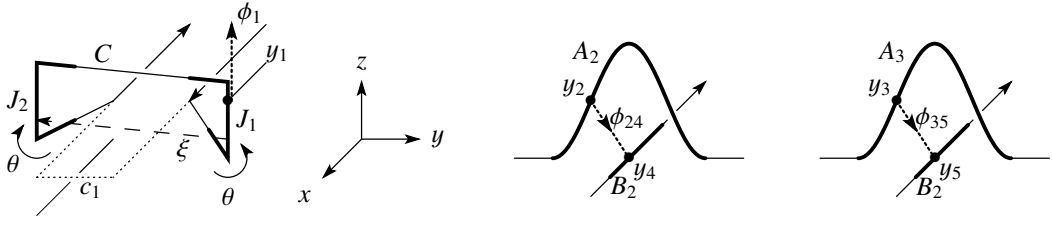}
\caption{The configuration that can non-trivially contribute to $I(X_3)$}
\label{fig:I(X_3)}
\end{figure}
Let $J_s$ ($s=1,2$) be the unit intervals identified with those on $C$ drawn with thick curves in Figure~\ref{fig:I(X_3)}.
We write $p_*\FH_{\wtf}(\theta)$ as $f_{\theta}$ for short.
Define $\phi_1\colon I_{c_1}\times J_s\to S^2$ ($s=1,2$), $\phi_{24}\colon A_2\times B_2\to S^2$ and $\phi_{35}\colon A_3\times B_3\to S^2$ by
\begin{equation}
 \phi_1(\theta,t)\coloneqq\frac{f_{\theta}'(t)}{\abs{f_{\theta}'(t)}},
 \quad
 \phi_{ij}(t,u)\coloneqq\frac{f(u)-f(t)}{\abs{f(u)-f(t)}},\quad(i,j)=(2,4),(3,5).
\end{equation}
Then
\begin{equation}\label{eq:I(X_3)_potentially_nonzero}
 \int_E\overline{p\circ\FH_{\wtf}}{}^*\omega_{X_3}=\int_{I_{c_1}\times(J_1\sqcup J_2)}\phi_1^*\vol\int_{A_2\times B_2}\phi_{24}^*\vol\int_{A_3\times B_3}\phi_{35}^*\vol.
\end{equation}

Define the diffeomorphisms $\xi\colon J_1\to J_2$ and $\eta\colon\R^3\to\R^3$ by
\begin{equation}
 \xi(t)=1-t,
 \quad
 \eta(x,y,z)\coloneqq(-x,y,-z).
\end{equation}
Then, with respect to the coordinates of $\R^3$ shown in Figure~\ref{fig:I(X_3)}, the following diagram commutes;
\begin{equation}
\begin{split}
 \xymatrix{
 I_{c_1}\times J_1\ar@{->}[r]^-{\phi_1}\ar@{->}[d]_-{\id\times\xi} & S^2\ar@{->}[d]^-{\eta} \\
 I_{c_1}\times J_2\ar@{->}[r]^-{\phi_1} & S^2\ar@{}[ul]|-{\circlearrowright}
 }
\end{split}
\end{equation}
and since $\xi$ reverses the orientation and $\eta$ preserves the orientation, we have
\begin{equation}
 \int_{I_{c_1}\times I_2}\phi_1^*\vol=-\int_{I_{c_1}\times I_1}\phi_1^*\vol
 \quad
 \text{and hence}
 \quad
 \int_{I_{c_1}\times(I_1\sqcup I_2)}\phi_1^*\vol=\sum_{s=1,2}\int_{I_{c_1}\times I_s}\phi_1^*\vol=0.
\end{equation}
Thus \eqref{eq:I(X_3)_potentially_nonzero} is zero.
\end{proof}

Thus we only need to compute the alternating sums of the integrations of $I(X_1)$ and $I(X_2)$ in the limit $h_1,h_2,h_3\to 0$.

\subsubsection{Computation of $I(X_1)$}
The following two subspaces of $E_{1;c_j,c_j'}$ ($j=1,2,3$) do not essentially contribute to the alternating sum of $I(X_1)$.
\begin{itemize}
\item
	The subspace where the arc near the left-most crossing moving in the FH-move contains no point; because the integrals on the subspace are the same for $\epsilon_j=+1$ and $\epsilon_j=-1$ and they cancel in the alternating sum.
\item
	The subspace where no edge joins points on $A_m$ and $B_m$ ($m=2,3$); because all the maps $\varphi_{\alpha}$ and hence the integrand $\omega_{X_1}$ can be defined even if $h_m=0$ and thus the cancellation similar to \eqref{eq:limit_is_zero_1} occurs.
\end{itemize}
Thus only the subspaces of types (1-a) and (1-b) consisting $(\theta,y)$ as shown in Figure~\ref{fig:I(X_1)} can essentially contribute to the integrations of $I(X_1)$.
\begin{figure}
\centering
\input{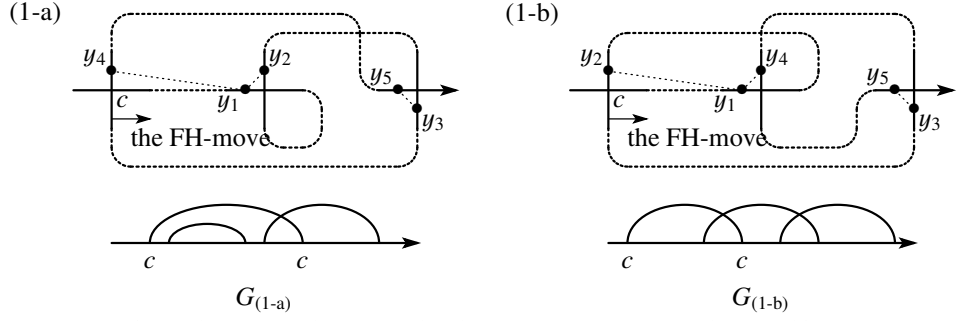}
\caption{Configurations essentially contributing to $I(X_1)$; they can exist only if the three crossings under consideration respect the Gauss diagrams $G_{\text{(1-a)}}$ or $G_{\text{(1-b)}}$}
\label{fig:I(X_1)}
\end{figure}
In both cases, the arc near the left-most crossing containing $y_2$ (case (1-a)) or $y_4$ (case (1-b)) moves to right in the FH-move, and when the arc comes over or under the middle crossing, the map $\varphi_{12}$ or $\varphi_{14}$ has its image in $\supp(\vol)$ and the integrand is not zero at that moment.

If three crossings $c_1,c_2,c_3$ under consideration respect one of the Gauss diagrams in Type I cycle (Figure~\ref{fig:Gauss_cycle_1}), then in the FH-cycle we meet the situation (1-a) in Figure~\ref{fig:I(X_1)} once, because the Gauss diagram $G_{\text{(1-a)}}$ appears once in Type I cycle.
If $c_1,c_2,c_3$ respect one of the Gauss diagrams in Type II cycle (Figure~\ref{fig:Gauss_cycle_2}), then in the FH-cycle we meet the situation (1-b) in Figure~\ref{fig:I(X_1)} twice, because the Gauss diagram $G_{\text{(1-b)}}$ appears twice in Type II cycle.
Otherwise we do not meet the situations (1-a) nor (1-b) and the integration vanishes.

\begin{prop}\label{prop:computation_I(X_1)}
We have
\begin{multline}\label{eq:I(X_1)_limit}
 \epsilon_1\epsilon_2\epsilon_3
 \sum_{c\in\{c_1,c_2,c_3\}}\int_{E_{1;c,c'}}\overline{p\circ\FH_{\wtf{}_{\epsilon_1,\epsilon_2,\epsilon_3}}}{}^*\omega_{X_1}\\
 =
 \begin{cases}
 1/8   & \text{if }c_1,c_2,c_3\text{ respect one of the Gauss diagrams in Type I cycle (Figure~\ref{fig:Gauss_cycle_1})},\\
 -1/4  & \text{if }c_1,c_2,c_3\text{ respect one of the Gauss diagrams in Type II cycle (Figure~\ref{fig:Gauss_cycle_2})},\\
 0     & \text{otherwise}.
 \end{cases}
\end{multline}
\end{prop}
\begin{proof}
Consider the first case; we may assume that $c_1,c_2,c_3$ respect the Gauss diagram $G_{\text{(1-a)}}$.
Then only $E_{1;c_1,c_1'}$ can contain the configurations of type (1-a) and non-trivially contribute to the alternating sum of the integrations of $I(X_1)$.

Let $b\colon\R^1\to\R^1$ be an even function whose graph looks as in Figure~\ref{fig:I(X_1)-a}.
\begin{figure}
\centering
\input{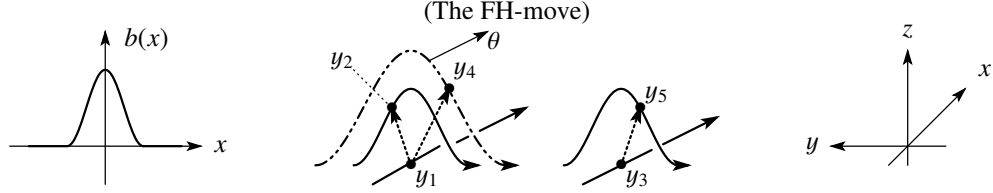}
\caption{Proof of Proposition~\ref{prop:computation_I(X_1)}; the case (1-a)}
\label{fig:I(X_1)-a}
\end{figure}
For $(\theta,x_1,\dots,x_5)\in\R^6$, consider $y_1,\dots,y_5\in\R^3$ given by
\begin{equation}\label{eq:five_points_(1-a)}
 y_1=(x_1,0,0),\ 
 y_2=(0,-\epsilon_2x_2,b(\epsilon_2x_2)),\ 
 y_3=(x_3,0,0),\ 
 y_4=(\theta,-\epsilon_1x_4,2b(\epsilon_1x_4/2)),\ 
 y_5=(0,-\epsilon_3x_5,b(x_5))
\end{equation}
and define $\varphi\colon\R^6\to(S^2)^{\times 3}$ by
\begin{equation}\label{eq:three_directions_1}
 \varphi(\theta,x_1,\dots,x_5)\coloneqq\left(\frac{y_2-y_1}{\abs{y_2-y_1}},\frac{y_5-y_3}{\abs{y_5-y_3}},\frac{y_4-y_1}{\abs{y_4-y_1}}\right).
\end{equation}
Then changing the valuables suitably, the left hand side of \eqref{eq:I(X_1)_limit} is equal to
\begin{equation}\label{eq:I(X_1)-(1-a)}
 \epsilon_1\epsilon_2\epsilon_3\int_{\R^6}\varphi^*(\vol^{\times 3}),
\end{equation}
where $\vol^{\times 3}=\pr_1^*\vol\wedge\pr_2^*\vol\wedge\pr_3^*\vol\in\Omega^6_{\DR}((S^2)^{\times 3})$.

Define $\Phi\colon\R^6\to(\R^2)^{\times 3}$ and $\psi_s\colon\R^2\to S^2$ ($s=1,2$) by respectively
\begin{gather}
 \label{eq:linear_diffeo} \Phi(\theta,x_1,\dots,x_5)\coloneqq((x_1,\epsilon_2x_2),(x_1-\theta,\epsilon_1x_4),(x_3,\epsilon_3x_5)),\\
 \psi_1(x,x')\coloneqq\frac{y'-y}{\abs{y'-y}},\quad
 \psi_2(x,x')\coloneqq\frac{y''-y}{\abs{y''-y}},
 \quad\text{where}\quad
 y\coloneqq(x,0,0),\ y'\coloneqq(0,-x,b(x)),\ y''=(0,-x',2b(x'/2)).
\end{gather}
Then $\Phi$ is a linear diffeomorphism whose determinant is $\epsilon_1\epsilon_2\epsilon_3$, and the following diagram is commutative;
\begin{equation}
\begin{split}
 \xymatrix{
  \R^6\ar@{->}[rr]^-{\varphi}\ar@{->}[rd]_-{\Phi} & \ar@{}[d]|-{\circlearrowright}& (S^2)^{\times 3}\\
   & (\R^2)^{\times 3}\ar@{->}[ru]_-{\psi_1^{\times 2}\times\psi_2}
 }
\end{split}
\end{equation}
Thus \eqref{eq:I(X_1)-(1-a)} is equal to
\begin{equation}
 (\epsilon_1\epsilon_2\epsilon_3)^2\left(\int_{\R^2}\psi^*_1\vol\right)^2\int_{\R^2}\psi_2^*\vol=\left(\frac{1}{2}\right)^3=\frac{1}{8},
\end{equation}
here $1/2$ appears by exactly the same reason as in Example~\ref{ex:I(X)_example}.

The second case that $c_1,c_2,c_3$ respect the Gauss diagram $G_{\text{(1-b)}}$ can be similarly computed, replacing
\begin{itemize}
\item
	\eqref{eq:five_points_(1-a)} and \eqref{eq:three_directions_1} respectively with
	\begin{gather}
	 y_1=(x_1,0,0),\ 
	 y_2=(\theta,-\epsilon_2x_2,b(\epsilon_2x_2/2)),\ 
	 y_3=(x_3,0,0),\ 
	 y_4=(0,-\epsilon_1x_4,b(x_4)),\ 
	 y_5=(0,-\epsilon_3x_5,b(x_5)),\\
	 \varphi(\theta,x_1,\dots,x_5)\coloneqq
	 \left(\frac{y_4-y_1}{\abs{y_4-y_1}},\frac{y_5-y_3}{\abs{y_5-y_3}},\frac{y_2-y_1}{\abs{y_2-y_1}}\right)
	\end{gather}
	(namely $y_2$ and $y_4$ are swapped),
	and
\item
	\eqref{eq:linear_diffeo} with
	\begin{equation}
	 \Phi(\theta,x_1,\dots,x_5)\coloneqq((x_1,\epsilon_2x_2),(x_1-\theta,\epsilon_1x_4),(x_3,\epsilon_3x_5)).
	\end{equation}
\end{itemize}
Then the determinant of $\Phi$ is $-\epsilon_1\epsilon_2\epsilon_3$, and because we meet the situation (1-b) twice in the FH-cycle, the left hand side of \eqref{eq:I(X_1)_limit} in this case is equal to
\begin{equation}
 -2(\epsilon_1\epsilon_2\epsilon_3)^2\left(\int_{\R^2}\psi^*_1\vol\right)^2\int_{\R^2}\psi^*_2\vol=-\frac{1}{4}.\qedhere
\end{equation}
\end{proof}

\subsubsection{Computation of $I(X_2)$}
The computation of $I(X_2)$ goes similarly to that of $I(X_1)$.
Only the subspaces of types (2-a) and (2-b) consisting of $(\theta,y)$ as shown in Figure~\ref{fig:I(X_2)} can essentially contribute to the alternating sum of the integrations of $I(X_2)$.
\begin{figure}
\centering
\input{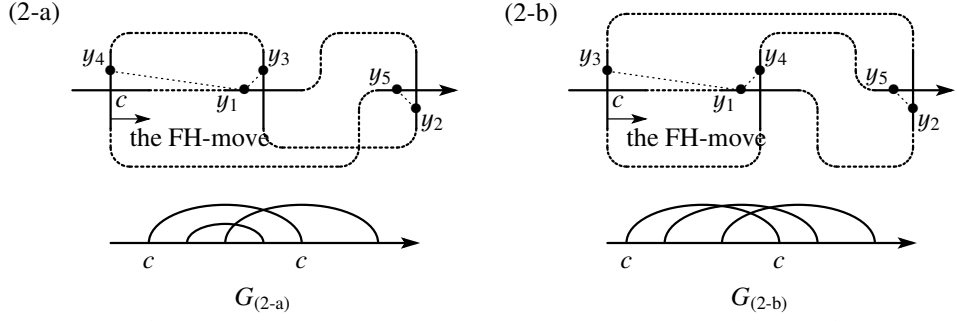}
\caption{Configurations essentially contributing to $I(X_2)$; they can exist only if the three crossings under consideration respect the Gauss diagrams $G_{\text{(2-a)}}$ or $G_{\text{(2-b)}}$}
\label{fig:I(X_2)}
\end{figure}
If three crossings $c_1,c_2,c_3$ under consideration respect one of the Gauss diagrams in Type II cycle (Figure~\ref{fig:Gauss_cycle_2}), then in the FH-cycle we meet the situation (2-a) in Figure~\ref{fig:I(X_1)} twice, because the Gauss diagram $G_{\text{(2-a)}}$ appears twice in Type II cycle.
If $c_1,c_2,c_3$ respect one of the Gauss diagrams in Type III cycle (Figure~\ref{fig:Gauss_cycle_3}), then in the FH-cycle we meet the situation (2-b) in Figure~\ref{fig:I(X_1)} six times, because the Gauss diagram $G_{\text{(2-b)}}$ appears six times in Type III cycle.

\begin{prop}\label{prop:computation_I(X_2)}
We have
\begin{multline}\label{eq:I(X_2)_limit}
 \epsilon_1\epsilon_2\epsilon_3
 \sum_{c\in\{c_1,c_2,c_3\}}\int_{E_{2;c,c'}}\overline{p\circ\FH_{\wtf{}_{\epsilon_1,\epsilon_2,\epsilon_3}}}{}^*\omega_{X_2}\\
 =
 \begin{cases}
 -1/4  & \text{if }c_1,c_2,c_3\text{ respect one of the Gauss diagrams in Type II cycle (Figure~\ref{fig:Gauss_cycle_2})},\\
 3/4   & \text{if }c_1,c_2,c_3\text{ respect one of the Gauss diagrams in Type III cycle (Figure~\ref{fig:Gauss_cycle_3})},\\
 0 & \text{otherwise}.
 \end{cases}
\end{multline}
\end{prop}
\begin{proof}
Consider the first case that $c_1,c_2,c_3$ respect the Gauss diagram $G_{\text{(2-a)}}$.
Then only $E_{2;c_1,c_1'}$ can contain the configurations of type (2-a) and non-trivially contribute to the integral.
\begin{figure}
\centering
\input{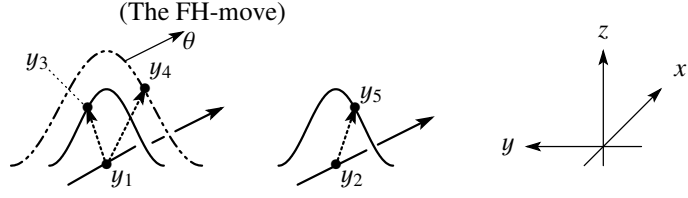}
\caption{Proof of Proposition~\ref{prop:computation_I(X_2)}}
\label{fig:I(X_2)-a}
\end{figure}
The proof of this case goes very similarly to the above ones;
we just need to replace
\begin{itemize}
\item
	\eqref{eq:five_points_(1-a)} and \eqref{eq:three_directions_1} respectively with
	\begin{gather}
	 y_1=(x_1,0,0),\ 
	 y_2=(x_2,0,0),\ 
	 y_3=(0,-\epsilon_2x_3,b(\epsilon_2x_3)),\ 
	 y_4=(\theta,-\epsilon_1x_4,2b(\epsilon_1x_4/2)),\ 
	 y_5=(0,\epsilon_3x_5,b(\epsilon_3x_5)),\\
	 \varphi(\theta,x_1,\dots,x_5)\coloneqq\left(\frac{y_3-y_1}{\abs{y_3-y_1}},\frac{y_5-y_2}{\abs{y_5-y_2}},\frac{y_4-y_1}{\abs{y_4-y_1}}\right),
	\end{gather}
\item
	\eqref{eq:linear_diffeo} with
	\begin{equation}
	 \Phi(\theta,x_1,\dots,x_5)\coloneqq((x_1,\epsilon_2x_3),(x_1-\theta,\epsilon_1x_4),(x_2,\epsilon_3x_5)).
	\end{equation}
\end{itemize}
Then $\Phi$ is a linear diffeomorphism with the determinant $-\epsilon_1\epsilon_2\epsilon_3$, and because we meet the situation (2-a) twice in the FH-cycle, the left hand side of \eqref{eq:I(X_2)_limit} in this case is equal to
\begin{equation}
 -2(\epsilon_1\epsilon_2\epsilon_3)^2\left(\int_{\R^2}\psi^*_1\vol\right)^2\int_{\R^2}\psi^*_2\vol=-\frac{1}{4}.
\end{equation}

Consider the second case that $c_1,c_2,c_3$ respect the Gauss diagram $G_{\text{(2-b)}}$.
The proof of this case goes very similarly to that of the case (1-b) in Proposition~\ref{prop:computation_I(X_1)}; we just need to replace
\begin{itemize}
\item
	\eqref{eq:five_points_(1-a)} and \eqref{eq:three_directions_1} respectively with
	\begin{gather}
	 y_1=(x_1,0,0),\ 
	 y_2=(x_2,0,0),\ 
	 y_3=(\theta,-\epsilon_1x_3,2b(\epsilon_1x_3/2)),\ 
	 y_4=(0,-\epsilon_2x_4,b(\epsilon_2x_4)),\ 
	 y_5=(0,\epsilon_3x_5,b(\epsilon_3x_5)),\\
	 \varphi(\theta,x_1,\dots,x_5)\coloneqq\left(\frac{y_4-y_1}{\abs{y_4-y_1}},\frac{y_5-y_2}{\abs{y_5-y_2}},\frac{y_3-y_1}{\abs{y_3-y_1}}\right),
	\end{gather}
\item
	\eqref{eq:linear_diffeo} with
	\begin{equation}
	 \Phi(\theta,x_1,\dots,x_5)\coloneqq((x_1-\theta,\epsilon_2x_3),(x_1,\epsilon_1x_4),(x_2,\epsilon_3x_5)).
	\end{equation}
\end{itemize}
Then $\Phi$ is a linear diffeomorphism with the determinant $\epsilon_1\epsilon_2\epsilon_3$, and because we meet the situation (2-b) six times in the FH-cycle, the left hand side of \eqref{eq:I(X_2)_limit} in this case is equal to
\begin{equation}
 6(\epsilon_1\epsilon_2\epsilon_3)^2\left(\int_{\R^2}\psi^*_1\vol\right)^2\int_{\R^2}\psi^*_2\vol=\frac{3}{4}.\qedhere
\end{equation}
\end{proof}

\begin{proof}[Proof of Theorem~\ref{thm:D^3v}]
Let $c_1,c_2$ and $c_3$ respect one of the Gauss diagrams in type I cycle (Figure~\ref{fig:Gauss_cycle_1}).
Then by \eqref{eq:X_1_and_X_2_determine_D^3v} and Propositions \ref{prop:computation_I(X_1)}, \ref{prop:computation_I(X_2)} we have
\begin{multline}
 D^3v(\wtf)=
 \sum_{k=1,2}a_k\sum_{c\in\{c_1,c_2,c_3\}}\sum_{\epsilon_1,\epsilon_2,\epsilon_3\in\{+1,-1\}}\epsilon_1\epsilon_2\epsilon_3\int_{E_{k;c,c'}}\overline{p\circ\FH_{\wtf{}_{\epsilon_1,\epsilon_2,\epsilon_3}}}{}^*\omega_{X_k}\\
 =(-2)\cdot\sum_{\epsilon_1,\epsilon_2,\epsilon_3\in\{+1,-1\}}\frac{1}{8}+1\cdot 0
 =-2\cdot\frac{1}{8}\cdot 8=-2.
\end{multline}
Next suppose that $c_1,c_2$ and $c_3$ respect one of the Gauss diagrams in type II cycle (Figure~\ref{fig:Gauss_cycle_1}).
Then by \eqref{eq:X_1_and_X_2_determine_D^3v} and Propositions~\ref{prop:computation_I(X_1)} and \ref{prop:computation_I(X_2)},
\begin{multline}
 D^3v(\wtf)
 =\sum_{k=1,2}a_k\sum_{c\in\{c_1,c_2,c_3\}}\sum_{\epsilon_1,\epsilon_2,\epsilon_3\in\{+1,-1\}}\epsilon_1\epsilon_2\epsilon_3\int_{E_{k;c,c'}}\overline{p\circ\FH_{\wtf{}_{\epsilon_1,\epsilon_2,\epsilon_3}}}{}^*\omega_{X_k}\\
 =(-2)\cdot\sum_{\epsilon_1,\epsilon_2,\epsilon_3\in\{+1,-1\}}\left(-\frac{1}{4}\right)+1\cdot\sum_{\epsilon_1,\epsilon_2,\epsilon_3\in\{+1,-1\}}\left(-\frac{1}{4}\right)=2.
\end{multline}
Lastly suppose that $c_1,c_2$ and $c_3$ respect one of the Gauss diagrams in type III cycle (Figure~\ref{fig:Gauss_cycle_1}).
Then
\begin{equation}
 D^3v(\wtf)
 =\sum_{k=1,2}a_k\sum_{c\in\{c_1,c_2,c_3\}}\sum_{\epsilon_1,\epsilon_2,\epsilon_3\in\{+1,-1\}}\epsilon_1\epsilon_2\epsilon_3\int_{E_{k;c,c'}}\overline{p\circ\FH_{\wtf{}_{\epsilon_1,\epsilon_2,\epsilon_3}}}{}^*\omega_{X_k}\\
 =(-2)\cdot0+1\cdot\sum_{\epsilon_1,\epsilon_2,\epsilon_3\in\{+1,-1\}}\frac{3}{4}=6.
\end{equation}
If $c_1,c_2$ and $c_3$ respect no Gauss diagram in three cycles, then $D^3v(\wtf)=0$.
\end{proof}

\subsection{An explicit description of $v$}\label{ss:explicit_form}
It is known (see \cite[p.~215]{JacksonMoffatt19} for example) that the space of the Vassiliev invariants for framed knots of order less than or equal to three are multiplically generated by the framing number $\lk$, the Casson invariant $v_2$ and the order three invariant $v_3$ (characterized by the conditions in Theorem~\ref{thm:main}).
Thus all the Vassiliev invariants of order less than or equal to three are linear combinations of
\begin{equation}
 \lk,\quad
 v_2,\quad
 \lk^2,\quad
 v_3,\quad
 \lk\cdot v_2,\quad
 \lk^3.
\end{equation}

\begin{lem}\label{lem:linking_term_vanish}
Our invariant $v$ is of the form $v=av_3+b\lk\cdot v_2+cv_2$ for some constants $a,b,c\in\R$.
\end{lem}
\begin{proof}
The value of $v$ on the trivial long knot $f_0(x)=(x,0,0)$ together with a framing number $w\in\Z$ is a linear combination of $w$, $w^2$ and $w^3$ because $v_2(f_0)=v_3(f_0)=0$.
But by the definition $p_*H_{(f_0,w)}$ is a constant loop of $\K$ for any $w\in\Z$.
Thus $v(f_0,w)=0$ for any $w\in\Z$, and the coefficients of $\lk$, $\lk^2$ and $\lk^3$ must be zero.
\end{proof}

Below we compute the constants $a,b,c$ in Lemma~\ref{lem:linking_term_vanish}.
We denote by $3_1^+$ and $3_1^-$ respectively the right-handed and the left-handed trefoil knots, by $4_1$ the figure eight knot.
By the formulas for $v_2$ and $v_3$ in \cite[Theorems 1, 2]{PolyakViro94} we have
\begin{equation}\label{eq:values_of_v2_and_v3}
 v_2(3_1^+)=v_2(3_1^-)=1,\quad
 v_2(4_1)=-1,\quad
 v_3(4_1)=0.
\end{equation}

\begin{prop}
We have $a=6$, $b=-1$.
\end{prop}
\begin{proof}
Consider the ``standard'' diagram of $3_1^+$ in Figure~\ref{fig:Gauss} and write it as $f=f_{+,+,+}$.
This can be seen as a framed long knot with framing number $+3$.
The diagram $f_{-,-,-}$ is $3_1^-$ with framing number $-3$ and all the other $f_{\epsilon_1,\epsilon_2,\epsilon_3}$ are trivial.
The Gauss diagram in Figure~\ref{fig:Gauss} appears in the Type III cycle in Figure~\ref{fig:Gauss_cycle_3} and $D^3v(f)=6$ by Theorem~\ref{thm:D^3v}.
Thus we have
\begin{equation}\label{eq:a+3b=3}
 6=D^3v(f)=\bigl(av_3(3_1^+)+b\cdot 3 +cv_2(3_1^+)\bigr)-\bigl(av_3(3_1^-)+b\cdot(-3) +cv_2(3_1^-)\bigr)
 =2a+6b,
\end{equation}
here the last equality holds by \eqref{eq:v_3_characterization} and \eqref{eq:values_of_v2_and_v3}.

Next consider the diagram of $4_1$ in Figure~\ref{fig:figure_eight}.
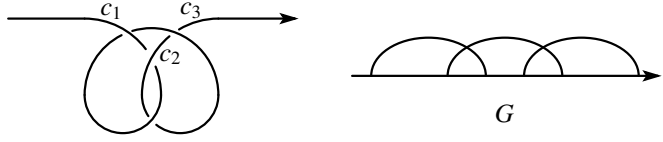
\begin{figure}
\centering
{\unitlength 0.1in%
\begin{picture}(34.0000,7.4500)(1.0000,-9.0000)%
%
\special{pn 13}%
\special{ar 500 700 400 400 4.7123890 5.6396842}%
%
\special{pn 13}%
\special{ar 700 700 200 200 6.2831853 3.1415927}%
%
\special{pn 13}%
\special{ar 850 700 350 350 4.4209322 6.2831853}%
%
\special{pn 13}%
\special{ar 1000 700 200 200 2.5535901 3.1415927}%
%
\special{pn 13}%
\special{ar 1200 700 400 400 3.1415927 4.0176507}%
%
\special{pn 13}%
\special{ar 500 700 400 400 5.9026789 6.2831853}%
%
\special{pn 13}%
\special{ar 850 700 350 350 3.1415927 4.2487414}%
%
\special{pn 13}%
\special{ar 1000 700 200 200 6.2831853 2.2455373}%
%
\special{pn 13}%
\special{ar 1200 700 400 400 4.1719695 4.7123890}%
%
\special{pn 13}%
\special{pa 1200 300}%
\special{pa 1600 300}%
\special{fp}%
\special{sh 1}%
\special{pa 1600 300}%
\special{pa 1533 280}%
\special{pa 1547 300}%
\special{pa 1533 320}%
\special{pa 1600 300}%
\special{fp}%
%
\special{pn 13}%
\special{pa 100 300}%
\special{pa 500 300}%
\special{fp}%
\put(7.0000,-3.0000){\makebox(0,0)[rb]{$c_1$}}%
\put(10.0000,-3.0000){\makebox(0,0)[lb]{$c_3$}}%
\put(9.5000,-5.0000){\makebox(0,0){$c_2$}}%
%
\special{pn 13}%
\special{pa 1900 600}%
\special{pa 3500 600}%
\special{fp}%
\special{sh 1}%
\special{pa 3500 600}%
\special{pa 3433 580}%
\special{pa 3447 600}%
\special{pa 3433 620}%
\special{pa 3500 600}%
\special{fp}%
%
\special{pn 13}%
\special{ar 2300 600 300 200 3.1415927 6.2831853}%
%
\special{pn 13}%
\special{ar 2700 600 300 200 3.1415927 6.2831853}%
%
\special{pn 13}%
\special{ar 3100 600 300 200 3.1415927 6.2831853}%
\put(27.0000,-8.0000){\makebox(0,0){$G$}}%
\end{picture}}%
\caption{A diagram of $g=4_1$ and the Gauss diagram that $c_1,c_2,c_3$ respect}
\label{fig:figure_eight}
\end{figure}
We write it as $g=g_{+,-,+}$, focusing on $c_1,c_2,c_3$.
This can be seen as a framed long knot with framing number $0$.
Then $g_{-,-,-}$ is the $3_1^-$ with framing number $-4$ and all the other $g_{\epsilon_1,\epsilon_2,\epsilon_3}$ are trivial.
The Gauss diagram $G$ in Figure~\ref{fig:figure_eight} appears in the Type II cycle in Figure~\ref{fig:Gauss_cycle_2} and $D^3v(g)=2$ by Theorem~\ref{thm:D^3v}.
Thus we have
\begin{equation}\label{eq:a+4b=2}
 2=D^3v(g)=-\bigl(av_3(4_1)+b\cdot 0+cv_2(4_1)\bigr)-\bigl(av_3(3_1^-)+b\cdot(-4)+cv_2(3_1^-)\bigr)
 =a+4b,
\end{equation}
here the last equality holds again by \eqref{eq:v_3_characterization} and \eqref{eq:values_of_v2_and_v3}.
Therefore $a=6$, $b=-1$ by \eqref{eq:a+3b=3} and \eqref{eq:a+4b=2}.
\end{proof}

\begin{prop}
We have $c=0$.
\end{prop}
\begin{proof}
As explained in \cite{Hatcher02}, for the ``standard'' diagram of $\wtf=3_1^+$ (with framing number $+3$, see Figure~\ref{fig:Gauss}), the FH-cycle $p_*\FH_{\wtf}$ is homologous to 3 times the Gramain cycle $G_f$ (see Remark~\ref{rem:Gramain}), where $f=p(\wtf)\in\K$.
This is because, as we can see in the Figure in \cite[p.~3]{Hatcher02}, each of the FH-moves on the crossings of $\wtf$ is the rotation around the long axis by degree $\pi$, and in the FH-cycle we perform six times of the FH-moves.
Thus
\begin{equation}\label{eq:v(3_1^+)}
 6v_3(3_1^+)-3v_2(3_1^+)+cv_2(3_1^+)=v(\wtf)=\int_{p_*\FH_{\wtf}}I(X)=3\int_{G_f}I(X).
\end{equation}
In \cite{K09} we have proved that the integration of $I(X)$ over $G_f$ is equal to $v_2(f)$.
Therefore \eqref{eq:v(3_1^+)} can be rewritten as
\begin{equation}
 6\cdot 1-3\cdot 1+c\cdot 1=3\cdot 1,
\end{equation}
and we have $c=0$.
\end{proof}

\section*{Acknowledgments}
The authors are deeply grateful to Arnaud Mortier for their invaluable comments and discussions.
They also express their appreciation to Thomas Fiedler for sharing the information about his 1-cocycles in his book.

\def\cprime{$'$}
\providecommand{\bysame}{\leavevmode\hbox to3em{\hrulefill}\thinspace}
\providecommand{\MR}{\relax\ifhmode\unskip\space\fi MR }
\providecommand{\MRhref}[2]{%
  \href{http://www.ams.org/mathscinet-getitem?mr=#1}{#2}
}
\providecommand{\href}[2]{#2}

\end{document}